\documentclass[11pt,reqno]{amsart}
\usepackage{amsmath} 
\usepackage{amssymb}

\usepackage{amsfonts}

\begin{document}

\title[Correlations for pairs of billiard trajectories]
{Correlations for pairs of periodic trajectories\\ for open billiards}

\author[V. Petkov]{Vesselin Petkov}
\address{Universit\'e Bordeaux I, Institut de Math\'ematiques de Bordeaux, 351, 
Cours de la Lib\'eration, 33405  Talence, France}
\email{petkov@math.u-bordeaux1.fr}
\author[L. Stoyanov]{Luchezar Stoyanov}
\address{University of Western Australia, School of Mathematics and
Statistics,  Perth, WA 6009,  Australia}
\email{stoyanov@maths.uwa.edu.au}
\thanks{The first author was partially supported by the ANR project NONAA}

\maketitle

\def\nexto{\kern -0.54em}
\newcommand{\C}{\protect\mathbb{C}}
\newcommand{\R}{\protect\mathbb{R}}
\newcommand{\Q}{\protect\mathbb{Q}}
\newcommand{\Z}{\protect\mathbb{Z}}
\newcommand{\N}{\protect\mathbb{N}}
\newtheorem{thm}{Theorem}
\newtheorem{prop}{Proposition}
\newtheorem{lem}{Lemma}
\newtheorem{deff}{Definition}
\def\DF{{\rm {I\ \nexto F}}}
\def\DR{{\rm {I\ \nexto R}}}
\def\DN{{\rm {I\ \nexto N}}}
\def\DZ{{\rm {Z \kern -0.45em Z}}}
\def\DC{{\rm\hbox{C \kern-0.83em\raise0.08ex\hbox{\vrule height5.8pt width0.5pt} \kern0.13em}}}
\def\DQ{{\rm\hbox{Q \kern-0.92em\raise0.1ex\hbox{\vrule height5.8pt width0.5pt}\kern0.17em}}}
\def\T{{\mathbb T}}
\def\S{{\mathbb S}}
\def\sn{{\S}^{n-1}}
\def\sN{{\S}^{N-1}}

\def\ep{\epsilon}
\def\e{\emptyset}
\def\di{\displaystyle}
\def\sk{\smallskip}
\def\bs{\bigskip}
\def\ms{\medskip}
\def\dk{\partial K}
\def\kamin{\kappa_{\min}}
\def\kamax{\kappa_{\max}}

\def\saa{\Sigma_A^+}
\def\sa{\Sigma_A}
\def\scc{\Sigma^+_C}
\def\sbb{\Sigma_B^+}
\def\san{\Sigma^-_A}

\def\Prf{\mbox{\footnotesize\rm Pr}}
\def\Pr{\mbox{\rm Pr}}

\def\be{\begin{equation}}
\def\ee{\end{equation}}
\def\beqn{\begin{eqnarray}}
\def\eeqn{\end{eqnarray}}
\def\beqn*{\begin{eqnarray*}}
\def\eeqn*{\end{eqnarray*}}
\def\endofproof{{\rule{6pt}{6pt}}}

\def\i{{\bf i}}
\def\iii{{\bf \i}}
\def\iu{\underline{i}}
\def\ju{\underline{j}}
\def\ku{\underline{k}}
\def\pu{\underline{p}}
\def\xx{{\bf x}}
\def\yy{{\bf y}}
\def\zz{{\bf z}}

\def\dist{\mbox{\rm dist}}
\def\diam{\mbox{\rm diam}}
\def\pr{\mbox{\rm pr}}
\def\supp{\mbox{\rm supp}}
\def\Arg{\mbox{\rm Arg}}
\def\In{\mbox{\rm Int}}
\def\Im{\mbox{\rm Im}}
\def\Int{\mbox{\rm Int}}
\def\span{\mbox{\rm span}}
\def\con{\mbox{\rm const}}
\def\Con{\mbox{\rm Const}}
\def\spec{\mbox{\rm spec}\,}
\def\Re{\mbox{\rm Re}}
\def\ecc{\mbox{\rm ecc}}
\def\var{\mbox{\rm var}}
\def\conf{\mbox{\footnotesize\rm const}}
\def\Conf{\mbox{\footnotesize\rm Const}}
\def\mt{\Lambda}
\def\la{\langle}
\def\ra{\rangle}
\def\Lip{\mbox{\rm Lip}}

\def\ff{{\mathcal F}}
\def\kk{{\mathcal K}}
\def\cc{{\mathcal C}}
\def\mm{{\mathcal M}}
\def\nn{{\mathcal N}}
\def\ll{{\mathcal L}}
\def\uu{{\mathcal U}}
\def\ss{{\mathcal S}}
\def\oo{{\mathcal O}}
\def\rr{{\mathcal R}}
\def\pp{{\mathcal P}}
\def\gg{{\mathcal G}}
\def\vv{{\mathcal V}}
\def\aa{{\mathcal A}}
\def\ii{{\imath }}
\def\jj{{\jmath }}
\def\H{{\mathcal H}}
\def\tf{\tilde{f}}

\def\hy{\hat{y}}
\def\hla{\hat{\lambda}}
\def\hx{\hat{x}}
\def\hxi{\hat{\xi}}
\def\hu{\hat{u}}
\def\hf{\hat{f}}
\def\hg{\hat{g}}
\def\hd{\hat{d}}
\def\ho{\hat{\omega}}
\def\hD{\hat{\Delta}}
\def\hxi{\hat{\xi}}
\def\heta{\hat{\eta}}
\def\hQ{\hat{Q}}
\def\htau{\hat{\tau}}
\def\hpi{\hat{\pi}}
\def\hL{\widehat{L}}
\def\he{\hat{e}}
\def\hchi{\hat{\chi}}
\def\hju{\hat{\ju}}
\def\hh{\hat{h}}
\def\nv{\nabla \varphi}
\def\Od{\mbox{\rm Od}}

\def\Re{\R\mbox{\rm e}}
\def\Ref{\R\mbox{\footnotesize\rm e}}
\def\h{h_{\mbox{\rm\footnotesize top}}}

\def\tc{\tilde{c}}
\def\tx{\tilde{x}}
\def\ty{\tilde{y}}
\def\tz{\tilde{z}}
\def\tu{\tilde{u}}
\def\tv{\tilde{v}}
\def\ta{\tilde{a}}
\def\td{\tilde{d}}
\def\tf{\tilde{f}}
\def\tg{\tilde{g}}
\def\tl{\tilde{\ell}}
\def\tr{\tilde{r}}
\def\tt{\tilde{t}}
\def\ti{\tilde{i}}
\def\tp{\tilde{p}}
\def\tw{\tilde{w}}
\def\tth{\tilde{\theta}}
\def\tka{\tilde{\kappa}}
\def\tla{\tilde{\lambda}}
\def\tmu{\tilde{\mu}}
\def\tga{\tilde{\gamma}}  
\def\trho{\tilde{\rho}}
\def\tvar{\tilde{\varphi}}
\def\tF{\tilde{F}}
\def\tU{\tilde{U}}
\def\tV{\tilde{V}}
\def\tS{\tilde{S}}
\def\tC{\tilde{C}}
\def\tP{\widetilde{P}}
\def\tQ{\widetilde{Q}}
\def\tR{\widetilde{R}}
\def\ts{\tilde{\sigma}}
\def\tpi{\tilde{\pi}}
\def\tla{\tilde{\lambda}}
\def\tf{\tilde{f}}

\def\v{{\sf v}}
\def\piU{\pi^{(U)}}
\def\mtb{ \mt_{\dk}}
\def\sAA{\Sigma_{\aa}^+}
\def\sA{\Sigma_{\aa}}
\def\sAn{\Sigma_{\aa}^-}
\def\jj{{\bf j}}

\def\hz{\hat{z}}
\def\ka{\kappa}
\def\wuloc{W^u_{loc}}
\def\lip{\mbox{\footnotesize\rm Lip}}
\def\clip{C^{\lip}}
\def\mtb{ \mt_{\dk}}

\begin{abstract}
In this paper we prove two asymptotic estimates for pairs of closed trajectories
for open billiards similar to those established by Pollicott and Sharp \cite{kn:PoS2}
for closed geodesics on negatively curved compact surfaces. The first of these estimates
holds for general open billiards in any dimension. The more intricate second estimate
is established for open billiards satisfying the so called Dolgopyat type estimates. This class of
billiards includes all open billiards in the plane and open billiards in $\R^N$ ($N \geq 3$)
satisfying some additional conditions.\\
\end{abstract}

\:\:\:\:\:{\bf 2000 AMS Subject Classification:}\:Primary: 37D50, Secondary: 58J50\\

{\bf Key words:} open billiard, periodic reflecting rays, symbolic coding.

\section{Introduction}
\renewcommand{\theequation}{\arabic{section}.\arabic{equation}}
\setcounter{equation}{0}

In \cite{kn:PoS2} Pollicott and Sharp prove some interesting asymptotic estimates
concerning the distribution of lengths $\lambda(\gamma)$ of closed geodesics $\gamma$
on a compact surface $V$ of negative curvature. Given a finite symmetric set $S$ of generators
of the fundamental group $\pi_1(V)$, for each closed geodesic $\gamma$ on $V$ let
$|\gamma|$  denote the minimal number of elements of $S$ needed to write down an element of
$\pi_1(V)$ conjugate to $\gamma$. Given real numbers $a < b$, the first asymptotic estimate in 
\cite{kn:PoS2} concerns the number $\pi (n,[a,b])$ of all pairs $(\gamma, \gamma')$ of closed
geodesics with $|\gamma|, |\gamma'| \leq n$ and $a \leq \lambda(\gamma) - \lambda(\gamma') \leq b$,
and has the same form as that in Theorem 1 below except that the constant $h_0$ 
that appears in \cite{kn:PoS2} depends on the set of generators $S$. The second asymptotic  in  \cite{kn:PoS2} 
is much more delicate and involves a family of  intervals $I_n = [z+ \epsilon_n a, z+ \epsilon_n b]$, $z\in \R$,
where $\epsilon_n \to 0$  subexponentially (see the formula in Theorem 2 below which has the same form
as the one in \cite{kn:PoS2}). In the proof of this a crucial role is played by Dolgopyat's estimates \cite{kn:D}
which apply to any Anosov flow on a compact surface (and also to some Anosov flows on higher
dimensional compact manifolds), and in particular to geodesic flows on compact surfaces of negative curvature.

In this paper we prove similar asymptotic estimates for the billiard flow in the exterior of several strictly convex bodes in 
$\R^N$ ($N \geq 2$) having smooth boundaries  and satisfying the so called no eclipse condition (H) defined below. 
In this case there is a natural coding for the closed trajectories using configurations (admissible sequences) of convex bodies 
and the constant $h_0 >0$ that appears in the asymptotic formulae is just the topological entropy of the billiard ball map from
boundary to boundary, and $h_0$ depends only on the number of obstacles (see Sect. 4).
The first asymptotic estimate (Theorem 1) holds for any open billiard in any dimension.
As in \cite{kn:PoS2},  the second asymptotic (Theorem 2) relies heavily on Dolgopyat type estimates.  For open 
billiards  these estimates are available without any extra assumptions for $N = 2$ (\cite{kn:St2}) and under some additional 
conditions for  $N \geq 3$ (\cite{kn:St3}). On the other hand, Dolgopyat type estimates concern codings via Markov families, 
and therefore they are not readily applicable to the natural coding of billiard trajectories mentioned above. However a link 
between these two types of codings can  be established which turns out to be sufficiently convenient, so that Dolgopyat type 
estimates can be applied in the situation described above and this is one of the purposes of this work. 

Correlations for periods of periodic orbits have been studied earlier in the physical literature. It appears \cite{kn:Aal} was the first 
article in this area, where some conjectures were made and  numerical results for three chaotic systems were described. 
We refer to \cite{kn:PoS2} for other references concerning problems and results on correlations for periodic orbits
in the physical and mathematical literature.\\

We now proceed to state precisely the results in this paper.

Let $K$ be a subset of ${\R}^{N}$ ($N\geq 2$) of the form
$K = K_1 \cup K_2 \cup \ldots \cup K_{\kappa_0},$where $K_i$ are compact 
strictly convex disjoint domains in $\R^{N}$ with 
$C^r$ ($r \geq 3$) {\it boundaries} $\Gamma_i = \dk_i$ and $\kappa_0 \geq 3$. 
Set $\Omega = \overline{{\R}^N \setminus K}.$ 
Throughout this paper we assume that $K$ satisfies the following (no-eclipse) condition: 
$${\rm (H)} \quad \quad\qquad  
\begin{cases}
\mbox{\rm for every pair $K_i$, $K_j$ of different connected components 
of $K$ the convex hull of}\cr
\mbox{\rm $K_i\cup K_j$ has no
common points with any other connected component of $K$. }\cr
\end{cases}$$
With this condition, the {\it billiard flow} $\phi_t$ defined on the {\it cosphere bundle} $S^*(\Omega)$ 
in the standard way is called an open billiard flow.
It has singularities, however its restriction to the {\it non-wandering set} $\Lambda$ has only 
simple discontinuities at reflection points.   Moreover, $\Lambda$  is compact, $\phi_t$ is hyperbolic and transitive
on $\Lambda$, and  it follows from  \cite{kn:St1} that $\phi_t$ is  non-lattice and therefore by  a result
of  Bowen \cite{kn:B}, it is topologically weak-mixing on $\Lambda$.

Given a periodic billiard trajectory (ray)  $\gamma$ in $ \Omega$, 
denote by $d_{\gamma}$ the period (return time) of
$\gamma$, and by $T_{\gamma}$ the  primitive period (length) of $\gamma$.
For any configuration ${\bf j} = (j_1,j_2,...,j_m)$ with $j_{\mu} \in \{1,...,\kappa_0\},\: j_i \not= j_{i+1}$ and $|\jj| = m,$ there exists an unique 
periodic reflecting ray $\gamma$ with reflecting points on $\Gamma_{j_1},...,\Gamma_{j_m}$ following the configuration $\jj$ (see \cite{kn:I1}, 
\cite{kn:PS1}) and we set $|\gamma| = |\jj| = m.$ 
We denote by $m_{\gamma}$ the number of reflections of $\gamma$ and by 
$P_{\gamma}$ the linear Poincar\'e map related to $\gamma$ (see \cite{kn:PS1}). 
For $a < b$ consider
$$\pi(n, [a, b]) = \# \{(\gamma, \gamma'):\: |\gamma|, |\gamma'| \leq n,\: a \leq T_{\gamma} - T_{\gamma'} \leq b\}$$
and denote by $h_0 > 0$ the {\it topological entropy} of the billiard ball map related to $\phi_t$
which coincides with the topological entropy of the shift map $\sigma$ on a naturally defined symbol space $\sa$ 
(see Subsect. 2.2 for the notation). This implies $h_0 = \log \lambda$, where $\lambda > 1$ is the maximal 
positive simple eigenvalue of the matrix $A$, so $h_0$ depends only on the number  $\kappa_0$ of connected components
of $K$. Then we have 
$$\# \{ \gamma:\: |\gamma| \leq n\} \sim \frac{e^{h_0}}{e^{h_0} -1} \frac{e^{h_0 n}}{n},\: n \to +\infty.$$

Our first result is the following

\begin{thm} 
There exists $\beta > 0$ such that for any $a < b$ we have
$$\pi(n, [a, b]) \sim \frac{(b-a) e^{2h_0}}{(2\pi)^{1/2} \beta(e^{h_0} - 1)^2} \frac{e^{2h_0 n}}{n^{5/2}},\: n \rightarrow +\infty.$$
\end{thm}

This estimate is derived from Lemma 1 in  Sect. 4 below and the argument from the proof of the first result in \cite{kn:PoS2}. 

To obtain a more precise result we use the  Dolgopyat type estimates (3.14) from Subsection 3.3 below.
Consider a sequence of intervals $I_n(z) = [z + \epsilon_n a, z + \epsilon_n b]$, where $\epsilon_n \searrow 0$. We say that $\epsilon_n$ goes to 0 
{\it subexponentially} if $\limsup_{n \to +\infty} |\log \epsilon_n|/n = 0.$ 
Using the notation of Sect. 2 and 3, the second result in this paper is the following analogue of Theorem 2 in \cite{kn:PoS2}.

\begin{thm} Assume that the estimates $(\ref{eq:3.14})$ hold for the Ruelle operator $\ll_{itr -h_0}, \: |t| \geq t_0 > 1$, 
and that the strong stable and the strong unstable laminations $\{ W^s_{\ep}(x)\}_{x\in \mt}$ and
$\{ W^u_{\ep}(x)\}_{x\in \mt}$ are Lipschitz in $x \in \mt$. 
Then there exists $\beta > 0$ such that for any $a < b$ and for every sequence $\epsilon_n$ going to 0 subexponentially, we have
$$\lim_{n \to +\infty} \sup_{z \in \R} \Bigl| \frac{\beta n^{5/2}}{\epsilon_n e^{2h_0 n}} \pi(n, I_n(z)) - \frac{(b-a) e^{2h_0}}{(2\pi)^{1/2} (e^{h_0} - 1)^2} e^{-z^2/2\beta^2n}\Bigr| = 0.$$
In particular, for any fixed $z \in \R$ we get
$$\pi(n, I_n(z)) \sim \frac{(b-a) e^{2h_0}\epsilon_n}{(2\pi)^{1/2} \beta(e^{h_0} - 1)^2}\frac{e^{2h_0 n}}{n^{5/2}},\: n \to +\infty.$$
\end{thm}

Notice that the assumptions of Theorem 2 are satisfied for $N = 2$ without any additional geometric conditions. The central point in the proof of 
Theorem 2 is Lemma 4 in Sect. 4, where the Dolgopyat type estimates are used.
We should remark that Lemma 4 deals with  Ruelle transfer operators defined by means of the symbolic coding $\saa$ using  the 
connected components $K_j$ of $K$ (see Sect. 2 for the notation). This coding is very natural and most convenient for a variety 
of problems where the open billiard flow is involved. However $\saa$  does not have the properties of a coding by means of
a Markov family and therefore Dolgopyat type estimates do not automatically apply to it. One needs to make a
transition from one coding to the other, and  this is done in Sects. 2 and 3 below. As a result of this  one 
identifies a class of functions on the symbolic model $\saa$ to which the estimates (3.14) can be applied. Unfortunately, this class 
of functions is not of the form  ${\mathcal F}_{\theta}(\saa)$ for some $\theta > 0$, and it is very doubtful that Dolgopyat type estimates 
would apply to all functions in ${\mathcal F}_{\theta}(\saa)$.

This rather serious difficulty appears also in the analysis of the approximation of the cut-off resolvent of the Dirichlet Laplacian in \cite{kn:PS2},
where we deal with phases and amplitudes related to the connected components $K_j$ of $K$ (and not to Markov sections $R_i$) to build an 
approximation. To overcome it, we use a link between the Ruelle operators $L_{-s\tf + \tg}$ and $\ll_{-sr + hg}$ 
corresponding to each other via the transition from one coding to the other.
The dynamical zeta function  related to the billiard flow is independent of the choice of coding and by using thermodynamic formalism 
(see e.g. \cite{kn:PP}) one can show that the eigenvalues of the operators $L_{-s\tf + \tg}$ and $\ll_{-sr + \hg}$ coincide with their multiplicities. 
However this does not imply similar estimates for the norms of the iterations of these operators.  To get such estimates
we use the  explicit link between $L_{-s\tf + \tg}$ and $\ll_{-sr + \hg}$  established in  Proposition 4 below for a special class of functions. From the analysis in Sect. 4 we need to have estimates for the operators $L_{-s\tf - h_0}^n$ and this corresponds to the case when $\tg = -h_0.$ 
%%%% Tozi tekst po-dolu ne mi se vijda iasen. Tova koeto se tvardi za entropiata ne e viarno - tia ostava sashtata pri vsiako kodirane.
It is important to note that we may code the periodic rays using a Markov family, 
however different Markov families will lead to rather different symbolic models, 
possibly with different number of symbols. The coding using the connected components
of $K$ is very natural and has a clear geometrical and physical meaning.

In the proofs of Theorems 1 and 2 (see Sect. 4) we use the analytic arguments in \cite{kn:PoS2} with minor changes, 
so we omit the details. Most of the arguments below concern the natural  symbolic model for the open billiard flow.
On the other hand, our result in Sect. 3 shows that there exists a class of functions for which the Dolgopyat estimates hold 
for the Ruelle operator related to the coding with obstacles and this has been applied in \cite{kn:PS2}
in the analysis of the analytic continuation of the cut-off resolvent of the Laplacian.

In Sect. 5 we discuss some open problems related to correlations of pairs of periodic trajectories when $\epsilon_n$ goes to 0 faster than $e^{-h_0 n}$ 
and their relationship with some separation conditions.

\section{Symbolic codings}
\renewcommand{\theequation}{\arabic{section}.\arabic{equation}}
\setcounter{equation}{0}

In this section we compare the Ruelle transfer operators related to two different codings
of the billiard flow $\phi_t$ on $\mt$ -- the first of these is related to a Markov family for $\mt$, while the
second is using the boundary components $\dk_i$.

Fix a {\it large ball} $B_0$ containing $K$ in its interior. 
For any $x\in  \Gamma = \partial K$ we will denote by
$\nu(x)$ the {\it outward unit normal} to $\Gamma$ at $x$.

It follows from Lemma 3.1 in \cite{kn:I1} that there exist $\delta_1 > 0$ and 
$0 < d_0 < \frac{1}{2}\min_{i\neq j} \dist (K_i, K_j)$
such that for any $i = 1, \ldots,\kappa_0$, $x\in \Gamma_i$ and $\xi \in \sN$ with 
$0 \leq \langle \xi, \nu (x)\rangle \leq \delta_1$ and $\eta = \xi -2 \la \xi, \nu(x) \ra \nu(x)$, at least one of the rays $\{ x+ t\xi : t\geq 0\},\: \{x +t\eta,\: t \leq 0\}$  has no common points with $\cup_{j\neq i} B(K_j,d_0)$, where $$B(A,d_0) = \{ y+ w : y\in A, w\in \R^N , \|w\| \leq d_0\}\;.$$

Let $z_0 = (x_0,u_0)\in S^*(\Omega)$. Denote by
$$X_1(z_0), X_2(z_0), \ldots, X_{m}(z_0), \ldots\;$$
the successive {\it reflection points} (if any) of the {\it forward trajectory} 
$$\gamma_+(z_0) = \{ \pr_1(\phi_{t}(z_0)) : 0 \leq t \}\;.$$
Similarly, we will denote by
$\gamma_-(z_0)$ the {\it backward trajectory} determined by $z_0$ and by 
$$\ldots, X_{-m}(z_0),  \ldots, X_{-1}(z_0), X_0(z_0)$$
its backward reflection points (if any).
If $\gamma (z_0) = \gamma_+(z_0) \cup \gamma_-(z_0)$ is bounded 
(i.e. it has infinitely many reflection points both forwards and backwards), we will 
say that it {\it has an itinerary} $\eta = (\eta_j)_{j=\infty}^\infty$ (or that it follows
the configuration $\eta$) if $X_j(z_0)\in \partial K_{\eta_j}$ for all
$j \in \Z$. We will say that the itinerary $\eta$ is {\it admissible} if
$\eta_j \neq \eta_{j+1}$ for all $j$. 

The following is a consequence of the hyperbolicity of the billiard
flow in the exterior of $K$ and can be derived from
the works of Sinai on general dispersing billiards (\cite{kn:Si1},
\cite{kn:Si2})  and from Ikawa's papers on open billiards
(\cite{kn:I1}; see also \cite{kn:Bu}). In this particular form it
can be found in \cite{kn:Sj} (see also Ch. 10 in  \cite{kn:PS1}).\\

\begin{prop} There exist global constants $C > 0$ and $\alpha\in (0,1)$ such that for any
admissible  configuration $\ii =  (i_0,i_1, i_2, \ldots, i_m)$ and any
two  billiard trajectories in $\Omega$
with successive reflection points $x_0,x_1, \ldots,x_m$ and $y_0,y_1, \ldots,y_m$, both 
following the configuration $\ii$, we have
$$\| x_j - y_j \| \leq C \, (\alpha^j + \alpha^{m-j}) \quad, \quad 0 \leq j \leq m\;.$$
Moreover, $C$ and $\alpha$ can be chosen so that if 
$(x_0, (x_1-x_0)/\|x_1-x_0\|)$ and $(y_0,  (y_1-y_0)/\|y_1-y_0\|)$ belong to the same
unstable manifold of the billiard flow,  then
$$\| x_j - y_j \| \leq C \, \alpha^{m-j} \quad, \quad 0 \leq j \leq m\;.$$
\end{prop}

\ms

As a consequence of this one obtains that there is an one-to-one
correspondence between the bounded (in both directions) billiard
trajectories in $\Omega$ and the set of admissible itineraries
$\eta$. In particular this implies that the intersections of the
non-wandering set $\mt$ with cross-sections to the billiard flow
$\phi_t$ are Cantor sets.

For $x\in \Lambda$ and a sufficiently small $\epsilon > 0$ let 
$$W^s_\ep(x) = \{ y \in S^*(\Omega) : d (\phi_t(x),\phi_t(y)) \leq \epsilon \: \mbox{\rm for all }
\: t \geq 0 \; , \: d (\phi_t(x),\phi_t(y)) \to_{t\to \infty} 0\: \}\; ,$$
$$W^u_\ep(x) = \{ y \in S^*(\Omega) : d (\phi_t(x),\phi_t(y)) \leq \epsilon \: \mbox{\rm for all }
\: t \leq 0 \; , \: d (\phi_t(x),\phi_t(y)) \to_{t\to -\infty} 0\: \}$$
be the (strong) {\it stable} and {\it unstable manifolds} of size $\epsilon$, and let
$E^u(x) = T_x W^u_{\ep}(x)$ and $E^s(x) = T_xW^s_{\ep}(x)$. 
For any $A \subset S^*(\Omega)$ and  $I \subset \R$  denote
$\phi_I(A) = \{\; \phi_t(y)\; : \; y\in A, t \in I \; \}.$

It follows from the hyperbolicity of $\mt$  (cf. e.g. \cite{kn:KH}) that if  
$\epsilon > 0$ is sufficiently small,
there exists $\delta > 0$ such that if $x,y\in \mt$ and $d (x,y) < \delta$, 
then $W^s_\ep(x)$ and $\phi_{[-\ep,\ep]}(W^u_\ep(y))$ intersect at exactly 
one point $[x,y ] \in \mt$. That is, $[x,y]\in W^s_\ep(x)$ and there exists a unique $t\in [-\ep, \ep]$ 
such that
$\phi_t([x,y]) \in W^u_\ep(y)$. Setting $\Delta(x,y) = t$, one defines the
so  called {\it temporal distance function} $\Delta(x, y).$ 
For $x, y\in \mt$ with $d (x,y) < \delta$, define
$$\pi_y(x) = [x,y] = W^s_{\ep}(x) \cap \phi_{[-\ep,\ep]} (W^u_{\ep}(y))\;.$$
Thus, for a fixed $y \in \mt$, $\pi_y : W \longrightarrow \phi_{[-\ep,\ep]} (W^u_{\ep}(y))$ is the
{\it projection} along local stable manifolds defined on a small open 
neighborhood $W$ of $y$ in $\mt$.

\ms

\subsection{Coding via a Markov family}

Given $E\subset \mt$ we will denote by $\Int_{\mt}(E)$ and $\partial_{\mt} E$  the {\it interior} 
and the {\it boundary} of the subset $E$ of $\mt$ in the topology of $\mt$, and by $\diam(E)$ the
{\it diameter} of $E$. We will say that
$E$ is an {\it admissible subset} of $W^u_{\ep}(z) \cap \mt$ ($z\in \mt$) if $E$ coincides with the closure
of its interior in $W^u_\ep(z) \cap \mt$. Admissible subsets of $W^s_\ep(z) \cap \mt$ are defined similarly.
Following \cite{kn:D}, a subset $R$ of $\mt$ will be called a {\it rectangle} if it has the form
$$R = [U,S] = \{ [x,y] : x\in U, y\in S\}\;,$$ 
where $U$ and $S$ are admissible 
subsets of $W^u_\ep(z) \cap \mt$ and $W^s_\ep(z) \cap \mt$, respectively, for some $z\in \mt$. 
For such $R$, given $\xi = [x,y] \in R$, we will denote $W^u_R(\xi) = \{ [x',y] : x'\in U\}$ and
$W^s_R(\xi) = \{[x,y'] : y'\in S\} \subset W^s_{\ep'}(x)$.

Let $\rr = \{ R_i\}_{i=1}^k$ be a family of rectangles with $R_i = [U_i  , S_i ]$,
$U_i \subset W^u_\ep(z_i) \cap \mt$ and $S_i \subset W^s_\ep(z_i)\cap \mt$, respectively, 
for some $z_i\in \mt$.   Set 
$$R =  \cup_{i=1}^k R_i\; .$$
The family $\rr$ is called {\it complete} if  there exists $T > 0$ such that for every $x \in \mt$,
$\phi_{t}(x) \in R$ for some  $t \in (0,T]$.  The {\it Poincar\'e map} $\pp: R \longrightarrow R$
related to a complete family $\rr$ is defined by $\pp(x) = \phi_{\tau(x)}(x) \in R$, where
$\tau(x) > 0$ is the smallest positive time with $\phi_{\tau(x)}(x) \in R$.
The function $\tau$  is called the {\it first return time} 
associated with $\rr$. Notice that  $\tau$ is constant on each of the set $W_{R_i}^s(x)$, $x \in R_i$. 
A complete family $\rr = \{ R_i\}_{i=1}^k$ of rectangles in $\mt$ is called a 
{\it Markov family} of size $\chi > 0$ for the  flow $\phi_t$ if $\diam(R_i) < \chi$ for all $i$ and: 

\ms

(a)  for any $i\neq j$ and any $x\in \Int_{\mt}(R_i) \cap \pp^{-1}(\Int_{\mt}(R_j))$ we have   
$$\pp(\Int_{\mt} (W_{R_i}^s(x)) ) \subset \Int_{\mt} (W_{R_j}^s(\pp(x)))\quad, \quad 
\pp(\Int_{\mt}(W_{R_i}^u(x))) \supset \Int_{\mt}(W_{R_j}^u(\pp(x)))\; ;$$

(b) for any $i\neq j$ at least one of the sets $R_i \cap \phi_{[0,\chi]}(R_j)$ and
$R_j \cap \phi_{[0,\chi]}(R_i)$ is empty.

\ms

The existence of a Markov family $\rr$ of an arbitrarily small size $\chi > 0$ for $\phi_t$
follows from the construction of Bowen \cite{kn:B} (cf. also  Ratner \cite{kn:Ra}). 
Taking $\chi$ sufficiently small, we may assume that each rectangle $R_i$ is `between
two boundary components' $\Gamma_{p_i}$ and $\Gamma_{q_i}$ of $K$, that is for any
$x\in R_i$, the first backward reflection point of the billiard trajectory $\gamma$ determined by $x$
belongs to $\Gamma_{p_i}$, while the first forward reflection point of $\gamma$ belongs to $\Gamma_{q_i}$.

Moreover, using the fact that the intersection of $\mt$ with each
cross-section to the flow $\phi_t$ is a Cantor set, we may assume that the Markov family
$\rr$ is chosen in such a way that

\ms

(c)  for any $i = 1, \ldots, k$ we have $\partial_\mt R_i  = \e$.

\ms

Finally, partitioning every $R_i$ into finitely many smaller rectangles, cutting $R_i$ along some unstable leaves, and removing
some rectangles from the family formed in this way, we may assume that

\ms

(d) for every $x\in R$ the billiard trajectory of $x$ from $x$ to $\pp(x)$ makes exactly one
reflection.

\ms

From now on we will assume that $\rr = \{ R_i\}_{i=1}^k$ is a fixed Markov family for  $\phi_t$
of size $\chi < \ep_0/2$ satisfying the above conditions (a)--(d). Set  
$$U = \cup_{i=1}^k U_i\;.$$
The {\it shift map} $\ts : U   \longrightarrow U$ is given by
$\ts  = \piU \circ \pp$, where $\piU : R \longrightarrow U$ is the {\it projection} along stable leaves.

Let $\aa = (\aa_{ij})_{i,j=1}^k$ be the matrix given by $\aa_{ij} = 1$ 
if $\pp(\Int_{\mt}(R_i)) \cap \Int_{\mt}(R_j) \neq  \e$ and $\aa_{ij} = 0$ otherwise. 
Consider the  symbol space 
$$\Sigma_\aa = \{  (i_j)_{j=-\infty}^\infty : 1\leq i_j \leq k, \aa_{i_j\; i_{j+1}} = 1
\:\: \mbox{ \rm for all } \: j\; \},$$
with the product topology and the {\it shift map} $\sigma : \Sigma_\aa \longrightarrow \Sigma_\aa$ 
given by $\sigma ( (i_j)) = ( (i'_j))$, where $i'_j = i_{j+1}$ for all $j$.
As in \cite{kn:B} one defines a natural map 
$$\Psi : \sA \longrightarrow R\;.$$
Namely, given any 
$ \ii = (i_j)_{j=-\infty}^\infty \in \Sigma_\aa$ there
is exactly one point $x\in R_{i_0}$ such that $\pp^j(x) \in R_{i_j}$ for
all integers $j$. We then set $\Psi( (i_j)) = x$. One checks that
$\Psi \circ \sigma = \pp\circ \Psi$ on $R$.
It follows from the condition (c) above that {\bf the map $\Psi$ is a  bijection}.   
Moreover $\Psi$ is Lispchitz if $\sA$ is 
considered with  the {\it metric} $d_\theta$ for some  appropriately chosen $\theta\in (0,1)$, where
$d_\theta(\xi,\eta) = 0$ if $\xi = \eta$ and $d_\theta(\xi,\eta) = \theta^m$ if
$\xi_i = \eta_i$ for $|i| < m$ and $m$ is maximal with this
property. Replacing $\theta$ by an appropriate 
$\rho \in (\theta,1)$, makes $\Psi^{-1}$ a Lipschitz map.

In a similar way one deals with the one-sided subshift of finite type
$$\sAA = \{  (i_j)_{j=0}^\infty : 1\leq i_j \leq k, \aa_{i_j\; i_{j+1}} = 1
\:\: \mbox{ \rm for all } \: j \geq 0\; \},$$
where the {\it shift map} $\sigma : \sAA \longrightarrow \sAA$ is defined in
a similar way: $\sigma( (i_j)) = ( (i'_j))$, where $i'_j = i_{j+1}$
for all $j \geq 0$. The metric $d_\theta$ on $\sAA$ is defined as above.
One checks that there exists a unique map $\psi : \sAA \longrightarrow U$ such that 
$\psi\circ \pi = \pi^{(U)}\circ \Psi$, where $\pi : \sA \longrightarrow \sAA$
is the {\it natural projection}.

Notice that the {\it roof function} $r : \sA \longrightarrow [0,\infty)$ defined by
$r (\xi) = \tau(\Psi(\xi))$ depends only on the
forward coordinates of $\xi\in \Sigma_\aa$. Indeed, if $\xi_+ = \eta_+$, where
$\xi_+ = (\xi_j)_{j=0}^\infty$, then for $x = \Psi(\xi)$ and $y = \Psi(\eta)$
we have $x, y\in R_i$ for $i = \xi_0= \eta_0$ and $\pp^j(x)$ and  $\pp^j(y)$ 
belong to the same $R_{i_j}$ for all $j\geq 0$. This implies that $x$ and $y$ belong to
the same local stable fiber in $R_i$ and therefore $\tau(x) = \tau(y)$.
Thus, $r (\xi) = r (\eta)$. So, we can define a {\it roof function} $r : \sAA \longrightarrow [0,\infty)$ 
such that $r\circ \pi =  \tau\circ \Psi$.  

Setting $r_n(\xi) = r(\xi) + r(\sigma \xi) + \ldots + r(\sigma^{n-1}\xi)$
for any integer $n \geq 1$ and any $\xi\in \saa$, we have the following

\begin{prop} There exists a bijection $\psi : \sAA \longrightarrow U$
such that $\psi\circ \sigma = \ts \circ \psi$ and a function 
$r : \sAA \longrightarrow [0,\infty)$ such that for any integer 
$n \geq 1$ and any $\xi\in \sAA$ we have 
$r_n(\xi) =  \tau_n(\psi(\xi))$, i.e. this is the length of the
billiard trajectory in $\Omega$ determined by $\psi(\xi)$ from $\psi(\xi)$ to
the $n$th intersection with a rectangle from $\rr$.
\end{prop}

Let $B(\sAA)$ be the {\it space of  bounded functions} 
$g : \sAA \longrightarrow \C$ with its standard norm  
$\|g\|_0 = \sup_{\xi\in \sAA} |g(\xi)|$. Given a function $g \in B(\sAA)$, the  {\it Ruelle transfer operator } 
$\ll_g : B(\sAA) \longrightarrow B(\sAA)$ is defined by
$$(\ll_gh)(\xi) = \sum_{\sigma(\eta) = \xi} e^{g(\eta)} h(\eta)\;.$$

Let $\ff_\theta(\saa)$ denote the space of $d_\theta$-Lipschitz functions $g : \sAA \longrightarrow \C$ 
with the norm
$$\|| f\||_\theta = \|f\|_0 + \|f\|_\theta\;,$$
where 
$$\|f\|_\theta = \sup \left\{ \frac{|f(\xi) - f(\eta)|}{d_\theta(\xi,\eta)} : \xi, \eta\in \sAA\;, \; \xi \neq \eta\right\}\;.$$
If $g \in \ff_\theta(\sAA)$, then  $\ll_g$ preserves the space $\ff_\theta(\sAA)$.

\ms

\subsection{Coding using boundary components}

Here we assume that  $K$ is as in Sect. 1.
Denote by $A$ the $\ka_0\times \ka_0$ matrix with entries $A(i,j) = 1$ if $i \neq j$ and
$A(i,i)= 0$ for all $i$, and set
$$\sa = \{ (\ldots, \eta_{-m}, \ldots, \eta_{-1}, \eta_0,\eta_1,  \ldots, \eta_{m},  \ldots ) : 
1\leq \eta_j \leq \ka_0,\:\eta_j \in \N,\:\:\: \eta_{j} \neq \eta_{j+1}\:\: \mbox{\rm for all }\: j \in \Z\}\;,$$
$$\saa = \{  (\eta_0, \eta_1, \ldots, \eta_{m},  \ldots ) : 1\leq \eta_j \leq \ka_0,\:\eta_j \in \N,\:\:\:
\eta_{j} \neq \eta_{j+1}\:\: \mbox{\rm for all }\: j \geq 0\}\;,$$
$$\san = \{ (\ldots, \eta_{-m},  \ldots, \eta_{-1}, \eta_0) : 1\leq \eta_j \leq \ka_0,\:\eta_j \in \N,\:\:\:
\eta_{j-1} \neq \eta_j\:\: \mbox{\rm for all }\: j \leq 0\}\;.$$
%Let $\pr_1 : S^*(\Omega) = \Omega\times \sN \longrightarrow \Omega$ and
%$\pr_2 : S^*(\Omega)  \longrightarrow \sN$ be the natural {\it projections}. 
The shift operators\footnote{We keep the same notation as for $\sA$ and $\sAA$; it will be clear 
from the context which of these we mean in each particular case.}
$\sigma: \: \sa \longrightarrow \sa$ and $\sigma: \saa \longrightarrow \saa$ 
are defined as before. 

We will now define two important functions $f$ and $g$ on $\sa$. For the second one in particular
we need some preliminary information.

A {\it phase function} on an open set $\uu$ in $\R^N$ is a smooth ($C^r$) function
$\varphi : \uu \longrightarrow \R$ such that $\| \nabla \varphi \| = 1$ everywhere in $\uu$.
For $x\in \uu$ the level surface
$$\cc_\varphi(x) = \{ y\in \uu: \varphi(y) = \varphi(x)\}$$
has a unit normal field $\pm \nabla \varphi(y)$. 

The phase function $\varphi$ defined on $\uu$ is said to satisfy the  {\it condition (${\mathcal P}$)} on $\Gamma_j$ if:

(i) the normal curvatures of $\cc_\varphi$ with respect to the normal field $-\nabla \varphi$ are
non-negative at every point of $\cc_\varphi$;

(ii) $\di \uu^+(\varphi) =  \{ y+t\nv(y) : t\geq 0, y\in \uu\} \supset \cup_{i\neq j} K_i$.

\medskip
A natural extension of $\varphi$ on $\uu^+(\varphi)$ is obtained by setting
$\varphi( y+t\nv(y)) = \varphi(y) + t$ for $ t\geq 0$ and $y\in \uu$.

%Let $z_0 = (x_0,u_0)\in S^*(\Omega)$. For convenience we will assume that $x_0 \notin K$. Assume that
%the {\it backward trajectory} $\gamma_-(z_0)$ determined by $z_0$ is bounded, and let $\eta  \in \san$
%be its itinerary.

%Given $x\in \R^N$ and $\epsilon > 0$, by $B(x,\epsilon)$ we denote
%the {\it open ball} with center $x$ and radius $\epsilon$ in $\R^N$. 
%Fix  a large ball $B_0$ in $\R^N$ containing $K$ in its interior.

For any $\delta > 0$ and $V \subset \Omega$ denote by $S^*_\delta(V)$
{\it the set of those} $(x,u)\in S^*(\Omega)$ such
that $x \in V$ and there exist $y\in \Gamma$ and $t \geq 0$ with 
$y+t u = x$, $y+ s u \in \R^N \setminus K$ for all
$s\in (0,t)$ and $\langle u, \nu (y) \rangle \geq \delta$. 

Notice that the condition (H) implies the existence of a constant $\delta_0 > 0$ depending only on the obstacle $K$
such that any $(x,u) \in S^*(\Omega)$ whose backward and forward billiard trajectories both have common points with 
$\Gamma$ belongs to $S^*_{\delta_0}(\Omega)$.

The following proposition is derived by using  some tools from \cite{kn:I1} (see Proposition 4 in \cite{kn:PS2} 
for details).

\begin{prop} There exists a constant $\epsilon_0 > 0$ such that 
for any $z_0 = (x_0,u_0)\in S_{\delta_0}^*(\Omega \cap B_0)$ whose
backward trajectory $\gamma_-(z_0)$ has an infinite number of 
reflection points $X_j = X_j(z_0)$ ($j\leq 0$)
and  $\eta\in \san$ is its itinerary, the following hold:

(a) There exists  a smooth ($C^r$) phase function $\psi = \psi_\eta$ 
satisfying the condition (${\mathcal P}$) on $\uu = B(x_0, \epsilon_0)\cap \Omega$  such that
$\psi(x_0) = 0$, $u_0 = \nabla \psi(x_0)$, and such that for any $x\in C_\psi(x_0)\cap \uu^+(\psi)$
the billiard trajectory $\gamma_-(x, \nabla\psi(x))$ has an itinerary $\eta$
and therefore 
$d(\phi_t(x, \nabla \psi(x)), \phi_t(z_0) ) \to 0$ as $t \to -\infty\;.$
That is, 
$$\wuloc (z_0) = \{ (x, \nabla\psi(x)) : x\in C_\psi(x_0)\cap \uu^+(\psi) \}$$
is the local unstable manifold of $z_0$. 
%Moreover, for any $p \geq  1$  there exists a global constant $C_p > 0$ (independent of $z_0$ and $\eta$) such that
%\be
%$$\|\nabla \psi_\eta\|_{(p)} (\uu) \leq C_p\;.$$
%\ee

%\medskip

(b)  If $(y,v)\in S^*(\Omega \cap B_0)$ is such that $y\in C_\psi(x_0)$ and $\gamma_-(y,v)$
has the same itinerary $\eta$, then $v = \nabla \psi(y)$, i.e. $(y,v) \in \wuloc(z_0)$.

\end{prop}

\bs

Notice that $\wuloc(z_0)$ coincides locally near $z_0$ with $W^u_{\epsilon}(z_0)$ defined above.

In what follows we will use the notation $C_{\eta}(z_0) = C_\psi(x_0)$. Denote by $G_{\eta}(z_0)$ the 
{\it Gauss curvature} of $C_\eta(z_0)$ at $x_0$.

Given $\xi \in \sa$, let
$$\ldots, P_{-2}(\xi), P_{-1}(\xi), P_0(\xi), P_1(\xi), P_2(\xi) , \ldots$$
be the successive reflection points of the unique billiard trajectory in the exterior of $K$ such that
$P_j(\xi) \in K_{\xi_j}$ for all $j \in \Z$. Set
$$f(\xi) = \| P_0(\xi) - P_1(\xi)\|\;,$$
and define the map
$$\Phi : \sa \longrightarrow \mtb = \mt \cap S^*_\mt(\partial K)$$
by $\Phi(\xi) = (P_0(\xi), (P_1(\xi) - P_0(\xi))/ \|P_1(\xi) - P_0(\xi)\|)$. Then $\Phi$ is a bijection
such that $\Phi\circ \sigma = B \circ \Phi$, where $B : \mtb \longrightarrow \mtb$ is the
{\it billiard ball map}.

Next, set $\xi_- = (\ldots, \xi_{-m}, \xi_{-m+1}, \ldots, \xi_{-1}, \xi_0) \in \san$ and choose an arbitrary point $x_0$
on the segment $[P_0(\xi), P_1(\xi)]$ such that $z_0 = (x_0, u_0) \in S^*_{\delta_0}(\Omega)$, where
$u_0 = (P_1(\xi) - P_0(\xi))/\|P_1(\xi) - P_0(\xi)\|$. Let $C_{\xi_-}(z_0) = C_\psi(x_0)$ for some phase function $\psi$.
Setting $t_- = \|x_0 - P_0(\xi)\|$ and $t_+ = \|x_0 - P_1(\xi)\|$, consider the surfaces
$$C^-_{\xi}(P_0(\xi)) = \{ x- t_- \nabla\psi(x) : x\in C_{\xi_-}(z_0) \}\quad ,\quad
C^+_{\xi}(P_1(\xi)) = \{ x+ t_+ \nabla\psi(x) : x\in C_{\xi_-}(z_0) \}\;.$$
Clearly $C^-_{\xi}(P_0(\xi))$ is the surface passing through $P_0(\xi)$ obtained by  shifting $C_{\xi_-}(z_0)$ in free space $\R^N$,
$t_-$ units backwards along its normal field, while $C^+_{\xi}(P_1(\xi))$ is the surface passing through $P_1(\xi)$
obtained by  shifting $C_{\xi_-}(z_0)$ in free space $\R^N$, $t_+$ units forwards  along its normal field. Let
$G^-_{\xi}(P_0(\xi))$ and $G^+_{\xi}(P_1(\xi))$ be the {\it Gauss curvatures} of $C^-_{\xi}(P_0(\xi))$ at $P_0(\xi)$ and 
that of $C^+_{\xi}(P_1(\xi))$ at $P_1(\xi)$, respectively. Set
$$g(\xi) = \frac{1}{N-1} \, \ln \frac{G^+_{\xi}(P_1(\xi))}{G^-_{\xi}(P_0(\xi))}\;.$$
This defines a function $g : \sa \longrightarrow \R$.

Choosing appropriately $\theta \in (0,1)$, we can define as in Subsection 2.1 above the space ${\mathcal F}_{\theta}(\sa)$ and we get
$f, g  \in {\mathcal F}_{\theta}(\sa)$  (see e.g. \cite{kn:I1}).
By Sinai's Lemma (see e.g. \cite{kn:PP}), there exist functions $\tf, \tg \in \ff_{\sqrt{\theta}}(\sa)$ depending on future coordinates only and $\chi_f, \chi_g \in \ff_\theta(\sa)$ such that
$$f(\xi) = \tf(\xi) + \chi_f(\xi) - \chi_f(\sigma \xi) \quad , \quad 
g(\xi) = \tg(\xi) + \chi_g(\xi) - \chi_g(\sigma \xi)$$
for all $\xi \in \sa$.
As in  the proof of Sinai's Lemma, for any $k = 1, \ldots, \ka_0$ choose and fix an arbitrary sequence 
$\eta^{(k)} = (\ldots, \eta^{(k)}_{-m}, \ldots, \eta^{(k)}_{-1}, \eta^{(k)}_0)\in \Sigma_a^-$ 
with $\eta^{(k)}_0 = k$. Then for any $\xi \in \sa$ (or $\xi \in \saa$) set
$$e(\xi) = (\ldots, \eta^{(\xi_0)}_{-m}, \ldots, \eta^{(\xi_0)}_{-1},
\eta^{(\xi_0)}_0= \xi_0, \xi_1, \ldots, \xi_m, \ldots)\in \sa \;.$$ 
Then we have
$$\chi_f(\xi) = \sum_{n=0}^\infty [ f(\sigma^n (\xi)) - f(\sigma^n e(\xi))]\;.$$
A similar formula holds for $\chi_g$.

As in section 2.1, given any function $V \in B(\saa)$, the  {\it Ruelle transfer operator } \\
$L_V : B(\saa) \longrightarrow B(\saa)$ is defined by
$$(L_V W)(\xi) = \sum_{\sigma(\eta) = \xi} e^{V(\eta)} W(\eta)\;.$$

\ms

\subsection{Another coding related to the Markov family}
\renewcommand{\theequation}{\arabic{section}.\arabic{equation}}
\setcounter{equation}{0}

Here we define another coding which uses the symbolic model $\sAA$. We will then define representatives
$\hf$ and $\hg$ of  the functions $f$ and $g$ on $\sAA$ and consider the corresponding Ruelle operators
$\ll_{-s \hf + \hg}$.

Let $V_0$ be the set of those $(p,u) \in S^*(\Omega)$ such that 
$p = q + t\, u$ and $(p,u) = \phi_t(q,u)$ for some 
$(q,u) \in S_{\dk}^*(\Omega)$ with $\la u , \nu(q) \ra  > 0$ and 
some $t \geq 0$. Clearly $V_0$ is an open subset of $S^*(\Omega)$ containing $\mt$.
Setting $\omega(p,u) = (q,u)$, we get a smooth map 
$\omega : V_0 \longrightarrow S_{\dk}^*(\Omega)$.

Consider the bijection $\ss  = \Phi^{-1}\circ \omega \circ \Psi : \sA \longrightarrow \sa$.
It induces a bijection $\ss  : \sAA \longrightarrow \saa$.
Moreover $\ss \circ \sigma = \sigma \circ \ss$.

Define the functions $f', g' : \sA \longrightarrow \R$ by $f'(\iu) = f(\ss(\iu))$ and
$g'(\iu) = g(\ss(\iu))$. 

Next, repeating a part of the previous subsection, for any $i = 1, \ldots, k$ choose
$$\hju^{(i)} = (\ldots, j_{-m}^{(i)}, \ldots,  j_{-1}^{(i)})$$ 
such that
$(\hju^{(i)}, i) \in \sA^-$. It is convenient to make this choice in such a way that $\hju^{(i)}$ {\bf corresponds to the
local unstable manifold} $U_i \subset \mt \cap W^u_\ep(z_i)$ (see the beginning of Subsection 2.1), i.e.
the backward itinerary of every $z \in U_i$ coincides with $\hju^{(i)}$.

Now for any $\iu = (i_0, i_1, \ldots) \in \sAA$ (or $\iu \in \sA$) set
$$\he(\iu) = (\hju^{(i_0)}; i_0,i_1, \ldots) \in \sA\;.$$
According to the choice of $\hju^{(i_0)}$, we have $\Psi(\he(\iu)) = \psi(\iu) \in U_{i_0}$.
(Notice that without the above special choice we would only have that 
$\Psi(\he(\iu))$ and  $\psi(\iu) \in U_{i_0}$ lie on the same stable leaf in $R_{i_0}$.)

Next, define
$$\hchi_f (\iu) = \sum_{n=0}^\infty \left[ f'(\sigma^n(\iu)) - f'(\sigma^n \, \he(\iu))\right]\quad , \quad \iu \in \sA\;,$$
and
$$\hchi_g (\iu) = \sum_{n=0}^\infty \left[ g'(\sigma^n(\iu)) - g'(\sigma^n \, \he(\iu)) \right] \quad , \quad \iu \in \sA\;.$$
As before, the functions $\hf, \hg : \sA \longrightarrow \R$ given by
$$\hf(\iu) = f'(\iu) - \hchi_f(\iu) + \hchi_f(\sigma \,\iu)\quad ,\quad
\hg(\iu) = g'(\iu) - \hchi_g(\iu) + \hchi_g(\sigma \,\iu)$$
depend on future coordinates only, so they can be regarded as functions on $\sAA$.

%There is a relationship between the operators $\ll_{-s \hf + \hg}^n$ and $L_{-s \tf + \tg}^n$,
%see Subsection 2.5 below.

\ms

\section{Relationship between Ruelle operators}
\subsection{Relationship between ${\mathcal L}_{-sr + \hg}$ and ${\mathcal L}_{-s\hf +\hg}$}
\def\ll{{\mathcal L}}
We will now describe a natural relationship between the operators 
$\ll_V : B(\sAA) \longrightarrow B(\sAA)$  and $L_v : B(\saa) \longrightarrow B(\saa)$, where
$V = v\circ \ss$.  Let $\rr = \{ R_i\}_{i=1}^k$ be a Markov family as in Subsection 2.1.  We define a map
$$\Gamma : B (\sa) \longrightarrow B(\sA)\;$$ 
by
\begin{equation} \label{eq:3.1}
\Gamma(v) = v\circ \Phi^{-1}\circ \omega \circ \Psi  = v \circ \ss\quad, \quad v \in B(\sa)\;.
\end{equation}
Since by property (d) of the Markov family, $\omega : R \longrightarrow \mtb$ is a bijection, it
follows that $\Gamma$ is a bijection and $\Gamma^{-1}(V) = V\circ \Psi^{-1}\circ \omega^{-1}\circ \Phi$.

Moreover $\Gamma$ induces a bijection $\Gamma : B (\saa) \longrightarrow B(\sAA)\;$.
Indeed, assume that $v \in B(\sa)$ depends on future coordinates only. Then $v\circ \Phi^{-1}$
is constant on local stable manifolds in $S^*_\mt(\Omega)$. Hence $v\circ \Phi^{-1} \circ \omega$
is constant on local stable manifolds on $R$, and therefore 
$\Gamma(v) = v\circ \Phi^{-1}\circ \omega \circ \Psi$ depends on future coordinates only.

Next, let $v, w\in B(\saa)$ and let $V = \Gamma(v)$, $W = \Gamma(w)$.
Given $\iu,\ju\in \sAA$ with $\sigma (\ju) = \iu$, setting $\xi = \ss(\iu)$ and $\eta = \ss (\ju)$, we
have $\sigma (\eta) = \xi$. Thus,
\begin{eqnarray*}
\ll_W V (\iu) 
& = & \sum_{\sigma(\ju) = \iu} e^{W(\ju)}\, V(\ju) = \sum_{\sigma(\ju) = \iu} e^{w(\ss (\ju))}\, v(\ss (\ju))\\
& = &  \sum_{\sigma(\eta) = \xi} e^{w(\eta)}\, v(\eta) = L_w v (\xi)
\end{eqnarray*}
for all $\iu \in \sAA$. This shows that
\be \label{eq:3.2}
(L_w v) \circ \ss = \ll_{\Gamma(w)} \Gamma(v) \;.
\ee

Notice that the functions $r : \sAA \longrightarrow [0,\infty)$ and $\tf : \saa \longrightarrow [0,\infty)$ do not correspond
to each other via $\ss$, and neither do $r$ and $\hf : \sAA \longrightarrow [0,\infty)$.
To compensate the difference between the latter two, define $\hla : \sA \longrightarrow [0,\infty)$ by
$\hla(\iu) = t > 0$, where for the point $x  = \Psi(\iu) \in R$ we have 
$\phi_{-t}(x) = \omega(x)\in S^*_{\dk}(\Omega)$.
With the same notation, define $\lambda (x) = t = \hla(\iu)$. This defines a function
$\lambda : R \longrightarrow [0,\infty)$  so that $\lambda \circ \Psi = \hla$.

Before continuing, notice that
\be \label{eq:3.3}
\hla(\he(\iu)) = \hla(\iu) - \hchi_f(\iu) \quad , \quad \iu \in \sA\;.
\ee
Indeed, given $\iu = (\ldots ; i_0,i_1, \ldots) \in \sA$, let $x = \Psi(\iu) \in R_{i_0}$ and $y  = \Psi(\he(\iu)) \in R_{i_0}$. 
Since, $\iu$ and $\he(\iu)$ have the same forward coordinates, it follows that $x$ and $y$ lie on the same local stable 
manifold. Thus, $\tau^n(x) = \tau^n(y)$ and so $r_n(\iu) = r_n(\he(\iu))$ for all integers $n \geq 1$.
Moreover, the $m$th reflection points of the billiard trajectories determined by $x$ and $y$ are  
$\Con \, \theta^m$-close for some global constants $\Con > 0$ and $\theta\in (0,1)$.  Thus,
$|\hla(\sigma^m (\iu))- \hla(\sigma^m(\he(\iu)) | \leq \Con \, \theta^m$, too (possibly with a different global constant $\Con > 0$).

Set $t = \hla(\iu)$ and $t' = \hla(\he(\iu))$. Given an integer $m \geq 1$, consider
$$A_m = \sum_{n = 0}^m \left[ f'(\sigma^n(\iu)) - f'(\sigma^n \, \he(\iu))\right]\;.$$
Then $\hchi_f(\iu) = \lim_{m\to \infty} A_m$.
%The above remarks shows that $|\hchi_f(\iu) - A_m| \leq \Con \, \theta^m$. 
Moreover,
\begin{eqnarray*}
A_m
& = & [r_m(\iu) + t - \hla(\sigma^m (\iu))] - [r_m(\he(\iu)) + t' - \hla(\sigma^m(\he(\iu))]\\
& = & (t-t') - [\hla(\sigma^m (\iu)) - \hla(\sigma^m(\he(\iu))] = \hla(\iu) - \hla(\he(\iu)) + O(\theta^m)\;.
\end{eqnarray*}
and letting $m \to \infty$, we obtain (\ref{eq:3.3}).

Next, for any $\iu \in \sA$ and any $m\geq 1$ we have
\begin{eqnarray*}
f'_m(\iu) 
& = & f'(\iu) + f'(\sigma\, \iu) + \ldots + f'(\sigma^{m-1}\iu)\\
& = & [\hf(\iu) + \hchi_f(\iu) - \hchi_f(\sigma\, \iu)] +  [\hf(\sigma\, \iu) + \hchi_f(\sigma\, \iu) - \hchi_f(\sigma^2\, \iu)] 
+ \ldots\\
&    &\quad  +  [\hf(\sigma^{m-1} \iu) + \hchi_f(\sigma^{m-1}\iu) - \hchi_f(\sigma^m\, \iu)] \\
& = & \hf_m(\iu) + \hchi_f (\iu) - \hchi_f(\sigma^m\iu)\;.
\end{eqnarray*}
Since $r_m(\iu) = f'_m(\iu) + \hla(\sigma^m\iu) - \hla(\iu)$, combining the above and (\ref{eq:3.3}), it follows that
$$r_m(\iu) =  \hf_m(\iu) + \hchi_f (\iu) - \hchi_f(\sigma^m\iu) +  \hla(\sigma^m\iu) - \hla(\iu)
= \hf_m(\iu) + \hla(\he (\sigma^m\iu)) - \hla(\he(\iu))\;.$$
Thus,
\be \label{eq:3.4}
r_m(\iu) = \hf_m(\iu) + \hla(\he (\sigma^m\iu)) - \hla(\he(\iu))\quad , \quad \iu\in \sA\:\:, \:\: m \geq 1\;.
\ee
Notice that $\hf_m(\iu) = \hf_m(\he(\iu))$, since $\hf$ depends on future coordinates only.

%Next, for any integer $n \geq 1$, $s\in \C$ and $\iu \in \sAA$ set
%$$\mu_n(\iu,s) =  -s\, \hf (\iu) + \hg(\iu) - s\, \hla(\he(\sigma^{n+1}\iu)) \;,$$
%and define the operator $\nn_{n,s}$ by
%\be \label{eq:3.5}
%(\nn_{n,s} v)(\iu) = \sum_{\sigma \, \ju = \iu} e^{\mu_n(\ju,s)} v(\ju) \quad , \quad \iu \in \sAA\;,
%\ee
%for any bounded function $v : \sAA \longrightarrow \C$.

Now we will find a relationship between the powers of the operators  
$\ll_{-s\,r + \hg}$ and $\ll_{-s\, \hf + \hg}$. Given $s\in \C$, consider the function $h_s : U \longrightarrow \R$ defined by
$$h_s(x) = e^{-s\, \lambda (x)}\;.$$
%Clearly $h_s$ is a {\bf Lipschitz function on } $U$. 
It gives rise to a function
$\hh_s : \sAA \longrightarrow \R$ defined by
\be \label{eq:3.5}
\hh_s(\iu) = h_s(\psi(\iu)) = h_s(\Psi(\he(\iu))) = e^{-s\, \hla (\he(\iu))} \quad , \quad \iu \in \sA\;.
\ee
(See the remark after the definition of $\he(\iu)$ in Sect. 2.3 above.)

It now follows from (\ref{eq:3.4}) and (\ref{eq:3.5}) that for any function  $V : \sAA \longrightarrow \C$, any $s\in \C$ and any integer $n \geq 1$ we have
\beqn*
\ll_{-s\, r+\hg}^{n+1} \left(\hh_s\cdot V\right)(\iu)
& = & \sum_{\sigma^{n+1}\ku = \iu} e^{-s\, r_{n+1}(\ku) + \hg_{n+1}(\ku)}\, e^{-s\, \hla (\he(\ku))}\, V(\ku)\\
& = & \sum_{\sigma^{n+1}\ku = \iu} e^{-s\, [\hf_{n+1}(\ku) + \hla(\he (\sigma^{n+1}\ku)) - \hla(\he(\ku))] + \hg_{n+1}(\ku) -s\, \hla (\he(\ku))}\, V(\ku)\\
& = & \sum_{\sigma^{n+1}\ku = \iu} e^{-s\, \hf_{n+1}(\ku) -s\, \hla(\he (\sigma^{n+1}\ku)) + \hg_{n+1}(\ku)}\, V(\ku)\\
& = &  e^{-s\, \hla(\he (\iu))}\; \sum_{\sigma^{n+1}\ku = \iu} e^{-s\, \hf_{n+1}(\ku)  + \hg_{n+1}(\ku)}\, V(\ku) 
= \hh_s(\iu)\, \cdot  \left(\ll_{-s\, \hf+\hg}^{n+1}V\right)(\iu)\;.
\eeqn*
Thus,
\be \label{eq:3.6}
\frac{1}{\hh_s}  \cdot \ll_{-s\, r+\hg}^{n} \left(\hh_s\cdot V\right) = \ll_{-s\, \hf+\hg}^{n }  \, V\;.
\ee

\ms

\subsection{Relationship between $\ll_{-s \hf + \hg}$ and $L_{-s\tf + \tg}$}

Here we will use arguments similar to these in Subsect. 3.1 to find a
relationship between the operators $\ll_{-s \hf + \hg}^n$ and $L_{-s \tf + \tg}^n$.

For any $p = 1, 2, \ldots,\kappa_0$ fix an arbitrary point $y_p \in \mt \cap S^*_{\delta_0}(\Omega)$ such that
$\eta^{(p)} \in \sa^-$ {\bf corresponds to the
unstable manifold} $W^u(y_p)$, i.e. the backward itinerary of every $z \in W^u(y_p)\cap V_0$ coincides with $\eta^{(p)}$.
Define $\kappa: \sAA \longrightarrow \R$ by 
$$\kappa(\iu) = [\Psi(\he(\iu)), y_{\xi_0}] \quad , \quad \xi = \ss(\iu) \in \saa\;.$$

We will now prove  that
\be \label{eq:3.7}
\hchi_f(\iu) - \chi_f(\ss (\iu)) =  \lambda(\kappa(\iu))- \hla(\he(\iu)) \quad , \quad \iu \in \sA\;.
\ee
It is important that the right-hand-side of (\ref{eq:3.7}) depends only on the future coordinates of $\iu$.
Indeed, given $\iu = (\ldots; i_0,i_1, \ldots) \in \sA$, set $x = \Psi(\he(\iu)) \in U_{i_0}$, $\xi = \ss (\iu)$,
$\xi_0 = p$ and $y  = \kappa(\iu) = [x,y_p]$.
Notice that $y \in W^s(x)$, and so the forward billiard trajectories of $y$ and $x$ converge. On the other hand,
$\phi_\tau (y) \in W^u(y_p)$ for some $\tau\in \R$ (in fact, $\tau = \Delta (x,y_p)$; see section 2 for the definition of $\Delta$), so the billiard trajectory determined by $y$ has
backward itinerary $\eta^{(p)}\in \sa^-$ and forward itinerary $(\xi_0,\xi_1,\ldots)\in \saa$, i.e. this
is the trajectory determined by $e(\xi)$. 

Since $y \in W^s(x)$, the $j$th reflection points $p_j$ and $q_j$ ($j\geq 0$)
of the billiard trajectories determined by $x$ and $y$,
respectively, are  $\Con \, \theta^j$-close for some global constants $\Con > 0$ and $\theta\in (0,1)$. 

Set $t = \hla(\he(\iu)) = \lambda(x)$ and $t' = \lambda(y) = \lambda(\kappa(\iu))$.
Given an integer $m \geq 1$, consider
$$A_m = \sum_{n = 0}^{m-1} \left[ f'(\sigma^n(\iu)) - f'(\sigma^n \, \he(\iu))\right] 
- \sum_{n = 0}^{m-1} \left[ f(\sigma^n(\xi)) - f(\sigma^n \, e(\xi))\right]\;.$$
Then $\hchi_f(\iu) - \chi_f(\ss (\iu))  = \lim_{m\to \infty} A_m$. Since $f'(\sigma^n(\iu))  = f(\sigma^n(\xi))$,
we have
\begin{eqnarray*}
A_m
= - \sum_{n = 0}^{m-1}f'(\sigma^n \, \he(\iu))
+ \sum_{n = 0}^{m-1} f(\sigma^n \, e(\xi))\;.
\end{eqnarray*}
The first sum in this expression is the length of the billiard trajectory determined by $x$ from $p_0$ till 
$p_m$, while the second is the length of the billiard trajectory determined by $y$ from $q_0$ till 
$q_m$. Since $\|q_m - p_m\| \leq \Con\, \theta^m$ and $y \in W^s(x)$, it now follows that
$A_m = -t+t' + O(\theta^m)$, and letting $m \to \infty$ proves (\ref{eq:3.7}).

In a similar way, using the definition of the function $g$, one derives that
\be \label{eq:3.8}
\hchi_g(\iu) - \chi_g(\xi) = \frac{1}{N-1}\, \ln G^-_{\he(\iu)}(P_0(\he(\iu)))
- \frac{1}{N-1}\, \ln G^-_{e(\xi)}(P_0(e(\xi)))
\quad , \quad \iu \in \sA\;.
\ee
Notice that $G^-_{\he(\iu)}(P_0(\he(\iu)))$ is the Gauss curvature of a shift of $U_{i_0}$ backwards along the
corresponding billiard trajectory at $P_0(\he(\iu))$ (so this is uniquely determined by the forward coordinates of $\iu$), while
$G^-_{e(\xi)}(P_0(e(\xi)))$ is the Gauss curvature of the shift of $W^u(y_{\xi_0})$ at
$P_0(e(\xi))$ (so again this is uniquely determined by the forward coordinates of $\iu$).

Given $\iu\in \sA$, set $\xi = \ss(\iu)$. Then (\ref{eq:3.7}) implies
\begin{eqnarray} \label{eq:3.9}
\hf_m(\iu) - \tf_m(\xi) 
& = & [f_m'(\iu) -\hchi_f (\iu) + \hchi_f(\sigma^m\iu)] - [f_m(\xi) - \chi_f (\xi) + \chi_f(\sigma^m\xi)]\nonumber\\
& = &  -[\hchi_f (\iu) - \hchi_f(\sigma^m\iu)] + [\chi_f (\xi) - \chi_f(\sigma^m\xi)]\;.
%& = & [ \hla(\he(\iu)) -  \hla(\kappa(\iu))] - [ \hla(\he(\sigma^m \iu)) -  \hla(\kappa(\sigma^m \iu))]\;.
\end{eqnarray}
Similarly,
\begin{eqnarray} \label{eq:3.10}
\hg_m(\iu) - \tg_m(\xi) = -[\hchi_g (\iu) - \hchi_g(\sigma^m\iu)] + [\chi_g (\xi) - \chi_g(\sigma^m\xi)]\;.
% = -\frac{1}{N-1}\, \ln \frac{G_{\he(\iu),0}(P_0(\he(\iu)))}{G_{e(\xi),0}(P_0(e(\xi)))}
%     + \frac{1}{N-1}\,\ln \frac{G_{\he(\sigma^m \iu),0}(P_0(\he(\sigma^m\iu)))}{
%G_{e(\sigma^m \xi),0}(P_0(e(\sigma^m \xi)))}\;.
\end{eqnarray}

Given $s\in \C$, consider the functions $d_s : U \longrightarrow \R$ and 
$\hd_s : \sAA \longrightarrow \R$ defined by
\be \label{eq:3.11}
d_s(\Psi(\iu)) = \hd_s(\iu) = e^{s\,[  \lambda(\kappa(\iu)) - \hla(\he(\iu)) ] 
- \frac{1}{N-1}\, \ln \frac{G^-_{\he(\iu)}(P_0(\he(\iu)))}{G^-_{e(\xi)}(P_0(e(\xi)))}} \quad , \quad \iu \in \sAA\;, \: \xi  = \ss(\iu)\;.
\ee
Notice that by (\ref{eq:3.7}) and (\ref{eq:3.8}), 
\be \label{eq:3.12}
\hd_s(\iu) = e^{s [\hchi_f(\iu) - \chi_f(\xi)]
- [\hchi_g(\iu) - \chi_g(\xi)] } \quad , \quad \xi = \ss(\iu).
\ee

Now for any function  $V : \sAA \longrightarrow \C$, any $s\in \C$ and
any integer $n \geq 1$, setting $\xi = \ss(\iu)$ and $\zeta = \ss(\ku)$ and using (\ref{eq:3.9}) and (\ref{eq:3.10}),  we get
\beqn*
L_{-s\, \tf+\tg}^{n} \left((\hd_s\circ \ss^{-1})\cdot (V \circ \ss^{-1})\right)(\xi)
& = & \sum_{\sigma^{n}\zeta = \xi} e^{-s\, \tf_{n}(\zeta) + \tg_{n}(\zeta)}\, 
\hd_s(\ku)\, V(\ku)\\
& = & \sum_{\sigma^{n}\ku = \iu} e^{-s\, \hf_{n}(\ku)  
- s [\hchi_f (\ku) - \hchi_f(\sigma^{n}\ku)] + s [\chi_f (\zeta) - \chi_f(\sigma^{n}\zeta)] }\\
&    & \:\:\: \times
e^{\hg_{n}(\ku) + [\hchi_g (\ku) - \hchi_g(\sigma^{n}\ku)] - [\chi_g (\zeta) - \chi_g(\sigma^{n}\zeta)]}\\
&    & \:\:\: \times e^{s [\hchi_f(\ku) - \chi_f(\zeta)] - [\hchi_g(\ku) - \chi_g(\zeta)] } \, V(\ku)\\
& = & \sum_{\sigma^{n}\ku = \iu} e^{-s\, \hf_{n}(\ku) + \hg_{n}(\ku)}
\times  e^{ s [\hchi_f(\iu) - \chi_f(\xi)] - [\hchi_g(\iu) - \chi_g(\xi)]}\, V (\ku)\\
& = & \hd_s(\iu) \cdot \left( \ll^{n}_{-s \hf + \hg} V \right) (\iu)\;.
\eeqn*

Combining the latter with (\ref{eq:3.6}), yields
$$L_{-s\, \tf+\tg}^{n} \left((\hd_s\circ \ss^{-1})\cdot (V\circ \ss^{-1})\right)(\xi)
=  \frac{\hd_s(\iu)}{\hh_s(\iu)}  \cdot  \ll_{-s\, r+\hg}^{n} (\hh_s\cdot V)  (\iu) \;.$$
Setting $u = (\hd_s\circ \ss^{-1})\cdot (V\circ \ss^{-1}) \in C(\saa)$ in the above yields the following
\begin{prop} For any $u \in C(\saa)$ we have
\be \label{eq:3.13}
\left(L_{-s\, \tf+\tg}^{n} u\right) (\ss(\iu))  = \frac{\hd_s(\iu)}{\hh_s(\iu)}  \cdot  
\ll_{-s\, r+\hg}^{n} \left(\frac{\hh_s}{\hd_s} \cdot (u\circ \ss) \right)  (\iu) \quad, \quad \iu \in \sAA\;.
\ee
In particular, the eigenvalues of $\ll_{-sr + \hg}$ and $L_{-s\tf + \tg}$ coincide with their multiplicities.
\end{prop}

\subsection{Dolgopyat type estimates}
To obtain Dolgopyat type estimates for the left-hand side of (\ref{eq:3.13}) we can use the Dolgopyat type estimates for 
$\ll_{-sr + \hg}$,  provided $u \circ \ss$ is determined by a Lipschitz function on $R$ (with respect to the distance on $R$ determined 
by the standard metric in $S^*(\Omega)$). 
To do so we also need to show that the functions $d_s$ and $h_s$ given by  (\ref{eq:3.11}) and (\ref{eq:3.5}) are Lipschitz on $U$. 

%We need some preliminary information.

% A finite string $\jj =  (j_0,j_1, j_2, \ldots, j_m)$ of numbers $j_i = 1, 2, \ldots, \ka_0$ will be called
%an {\it admissible  configuration} (of {\it length} $|\jj| = m+1$) if $j_i \neq j_{i+1}$ for all $i = 0,1,\ldots,m-1$. 
%We will say that a
%billiard trajectory $\gamma$ with successive reflection points $x_0,x_1, \ldots,x_m$
%{\it follows the configuration} $\jj$ if $x_i \in \Gamma_{j_i}$ for all $i =0,1,\ldots,m$.

\begin{lem}
Assume that  the strong stable and the strong unstable laminations $\{ W^s_{\ep}(x)\}_{x\in \mt}$ and
$\{ W^u_{\ep}(x)\}_{x\in \mt}$ are Lipschitz in $x \in \mt$. 
Then the  functions $d_s$ and $h_s$ are Lipschitz on $U$.
\end{lem}

\begin{proof}  Consider the function $d_s$ on $U_{\ti}$ for some fixed $\ti = 1, \ldots, k$. 
Given $\ti$, there exists $\tp = 1, \ldots, k_0$ such that $\pr_1(\omega(U_{\ti})) \subset \Gamma_{\tp}$.
Fix an arbitrary $\tz \in \mt \cap S^*_{\delta_0}(\Omega)$ with a backward itinerary $\eta^{(\tp)}\in \Sigma^-_A$ and
$\pr_1(\omega(\tz)) \in \Gamma_{\tp}$ and an arbitrary $\ty\in U_{\ti}$ with backward 
itinerary $\hju^{(\ti)} \in \Sigma^-_{\aa}$. Consider
the surface $C = C_{\eta^{(\tp)}} (\tz)$ (see Subsection 2.2 for the definition).
%and the surface $\tC = \pr_1(U_{\ti}) \subset \Omega$.

Notice that for any $\iu = (i_0,i_1, \ldots) \in \sAA$ with $i_0 = \ti$ and any $y\in U_{\ti}$, locally the surface
$C_{\he(\iu)} (y)$ coincides with $\pr_1(\phi_t(U_{\ti}))$ for some $t = t(y)$ which is a Lipschitz
function of $y \in U_{\ti}$ (in fact $t(y)$ extends smoothly to a neighborhood of $U_{\ti}$ in the local 
unstable manifold containing it). Thus, $G_{\he(\iu)} (y)$ is a Lipschitz function of $y \in U_{\ti}$. 
Similarly, for $\xi = (\xi_0, \xi_1, \ldots) \in \saa$ with $\xi_0 = \tp$ and $z\in \mt$ with
$\pr_1(\omega(z)) \in \Gamma_{\tp}$, $G_{e(\xi)} (z)$ is a Lipschitz function of $z$.

Next, let $x \in U_{\ti}$. Then $x = \psi(\iu)$ for some $\iu = (i_0, i_1, \ldots) \in \Sigma_{{\mathcal A}}$ with $i_0 = \ti$, 
and by the choice of $\he(\iu)$ (see Subsection 2.3), we have $x = \Psi(\he(\iu)) \in U_{\ti}$. Thus,
$\hla(\he(\iu)) = \lambda(x)$, and the definition of $\kappa$ gives
$\kappa (\iu) = [x,y_{\tp}]$. Moreover, $x = \Psi(\he(\iu))$ shows that $P_0(\he(\iu)) = \pr_1(\omega(x))$.
Finally, notice that $P_0(e(\xi)) = \pr_1(\omega([x,y_{\tp}]))$. Indeed, the point $z = [x,y_{\tp}]$ lies
on $W^s(x)$, so its forward itinerary in the model $\sAA$ is the same as that of $x = \psi(\iu)$, i.e. it
is $\iu$. Thus,  the forward itinerary of $z$ in the model $\saa$ is $\ss(\iu) = \xi$. On the other hand,
$\phi_s(z) \in W^u(y_{\tp})$ for some (small) $s\in \R$, so $z$ has the same backward itinerary in the
model $\sa$ as $y_{\tp}$, i.e. it is $\eta^{(\tp)}$. Thus, $z$ lies on the trajectory determined by 
$e(\xi) = (\eta^{(\tp)}; \xi)\in \sa$, and therefore 
$P_0(e(\xi)) = \pr_1(\omega(z)) =  \pr_1(\omega([x,y_{\tp}]))$.

It now follows from the above and (\ref{eq:3.11}) that
$$\ln d_s(x) = s\, [\lambda([x,y_{\tp}]) - \lambda(x)] 
- \frac{1}{N-1}\, \ln \frac{G^-_{\he(\iu)} (\pr_1(\omega(x)))}{G^-_{e(\xi)} (\pr_1(\omega([x,y_{\tp}])))}\;,$$
Since $\lambda : R \longrightarrow [0,\infty)$ is Lipschitz (it is actually smooth on a neighborhood of 
$\mt$ in $S^*(\Omega)$) and $[\cdot, \cdot]$ is uniformly Lipschitz, 
it follows that $d_s(x)$ is Lipschitz with $\Lip(d_s) \leq \Con\, |s|$ when $\Re(s)$ is bounded, and we can also write
$$\Lip(d_s) \leq \Con\, |\Im(s)| \quad , \quad s\in \C\:, \: |\Re(s)| \leq \Con\;.$$
The same argument applies to the function $h_s$.
\end{proof}

Denote by $\clip(U)$ the {\it space of Lipschitz functions} $v : U \longrightarrow \C$ 
For such $v$  let $\Lip(v)$ denote the {\it Lipschitz constant} of $v$, and for $t\in \R$, $|t| \geq 1$, define
$$\|v\|_{\lip,t} = \|v\|_0 + \frac{\Lip(v)}{|t|} \quad , \quad
\|v\|_0 = \sup_{x\in U} |v(x)|\;.$$
Let $P(F)$ denote the {\it topological pressure} of $F$ defined by
$$P(F) = \sup_{\mu \in {\mathcal M}_{\tilde{\sigma}}} \bigl[h(\mu) + \int_{\Sigma_{\mathcal A}^+} F\, d\mu ],$$
where ${\mathcal M}_{\tilde{\sigma}}$ is the set of all probability measures on $ \Sigma_{\mathcal A}^+$ invariant with respect 
to $\tilde{\sigma}$ and  $h(\mu)$ is the measure theoretic entropy of $\tilde{\sigma}$ with respect to $\mu.$

In the next section we will need an estimate for the iterations $L^n_{-s\tf -h_0}$.  
%To do so, according to (3.13), setting $\xi = \ss(\iu)$, we have to estimate
%$\di {\mathcal L}_{-sr - h_0}^n \Bigl( \frac{\hh_s}{\hd_s}    \Bigr)(\iu).$
Below we deal with more general situation when $\tg$ and $\hg$ are not necessarily constant functions. 
In particular, we study the case when  $g$ is defined as in Subsection 2.2 by Gauss curvatures at reflection points. This analysis is
motivated by applications related to the dynamical zeta function (see [PS2]). The case $g = h_0$ is covered by the
same argument.

We will apply Dolgopyat type estimates (\cite{kn:D}) 
established in the case of open billiard flows in \cite{kn:St2} 
for $N = 2$ and in \cite{kn:St3} for $N \geq 3$ under certain assumptions. 
%For $N \geq 3$ these estimates are obtained under the conditions ${\rm (P), (NF)}$ and ${\rm (ND)}$ and we refer to 
%\cite{kn:St3} for the precise definitions (see also Appendix B in \cite{kn:PS2}).
We are now going to state these assumptions in details.

%For $x\in \Lambda$ and a sufficiently small $\epsilon > 0$ let 
%$$W_{\ep}^s(x) = \{ y\in S^*(\Omega) : d (\phi_t(x),\phi_t(y)) \leq \epsilon \: \mbox{\rm for all }
%\: t \geq 0 \; , \: d (\phi_t(x),\phi_t(y)) \to_{t\to \infty} 0\: \}\; ,$$
%$$W^u_{\ep}(x) = \{ y\in S^*(\Omega) : d (\phi_t(x),\phi_t(y)) \leq \epsilon \: \mbox{\rm for all }
%\: t \leq 0 \; , \: d (\phi_t(x),\phi_t(y)) \to_{t\to -\infty} 0\: \}$$

The following {\it pinching condition}\footnote{It appears that in the proof of the Dolgopyat type estimates  in the case of 
open billiard flows (and some geodesic flows), one should be able to replace the condition (P) by 
just assuming Lipschitzness  of the stable and unstable laminations -- this will be the subject of 
some future work.} is one of the assumptions needed below:

\ms

\noindent
{\sc (P)}:  {\it There exist  constants $C > 0$ and $\alpha > 0$ such that for every $x\in \mt$
we have
$$\frac{1}{C} \, e^{\alpha_x \,t}\, \|u\| \leq \| d\phi_{t}(x)\cdot u\| \leq C\, e^{\beta_x\,t}\, \|u\|
\quad, \quad  u\in E^u(x) \:\:, t > 0 \;,$$
for some constants $\alpha_x, \beta_x > 0$ depending on $x$ but independent of $u$ with
$\alpha \leq \alpha_x \leq \beta_x $ and $2\alpha_x - \beta_x \geq \alpha$ for all $x\in \mt$.}

\ms

Notice that when $N=2$ this condition is always satisfied. For $N \geq 3$, (P) follows from 
certain estimates on the eccentricity of the connected components $K_j$ of $K$.  According to
general regularity results (\cite{kn:PSW}), (P) implies that $W^u_\ep(x)$ and $W^s_\ep(x)$
are Lipschitz in $x\in \mt$. 

Next, consider the following {\it non-flatness condition}:

\ms

\noindent
{\sc (NF)}:  {\it For every $x\in \mt$ there exists $\ep_x > 0$ such that there is no
$C^1$ submanifold $X$ of $W^u_{\ep_x}(x)$ of positive codimension with
$\mt \cap W^u_{\ep_x}(x) \subset X$.}

\ms

Clearly this condition is always satisfied if $N =2$, while for $N \geq 3$ it is at least generic. 
In the proof of the main result in \cite{kn:St3} this condition plays a technical role, and 
one would expect that a future refinement of the proof would remove it.

Next, we need some definitions from \cite{kn:St3}.
Given $z \in \mt$, let $\exp^u_z : E^u(z) \longrightarrow W^u_{\ep_0}(z)$  and
$\exp^s_z : E^s(z) \longrightarrow W^s_{\ep_0}(z)$ be the corresponding
{\it exponential maps}. A  vector $b\in E^u(z)\setminus \{ 0\}$ is called  {\it tangent to $\mt$} at
$z$ if there exist infinite sequences $\{ v^{(m)}\} \subset  E^u(z)$ and $\{ t_m\}\subset \R\setminus \{0\}$
such that $\exp^u_z(t_m\, v^{(m)}) \in \mt \cap W^u_{\ep}(z)$ for all $m$, $v^{(m)} \to b$ and 
$t_m \to 0$ as $m \to \infty$.  It is easy to see that a vector $b\in E^u(z)\setminus \{ 0\}$ is  tangent to $\mt$ at
$z$ iff there exists a $C^1$ curve $z(t)$ ($0\leq t \leq a$) in $W^u_{\ep}(z)$ for some $a > 0$ 
with $z(0) = z,\: \dot{z}(0) = b$, and $z(t) \in \mt$ for arbitrarily small $t >0$. In a similar way one
defines tangent vectors to $\mt$ in $E^s(z)$.  

Denote by $d\alpha$ the standard symplectic form on $T^*(\R^N) = \R^N \times \R^N$. 
The following condition says that $d\alpha$ is in some sense 
non-degenerate on the `tangent space' of $\mt$ near some of its points:

\ms

\noindent
{\sc (ND)}:  {\it There exist $z_0\in \mt$, $\ep > 0$  and  $\mu_0 > 0$ such that
for any $\hz \in \mt \cap W^u_{\ep}(z_0)$ and any unit vector $b \in E^u(\hz)$ tangent to 
$\mt$ at $\hz$ there exist $\tz \in \mt \cap W^u_{\ep}(z_0)$ arbitrarily close to $\hz$ and  
a unit vector $a \in E^s(\tz)$ tangent to $\mt$ at $\tz$ with $|d\alpha(a,b) | \geq \mu_0$.}

\bs

Clearly when $N = 2$ this condition is always satisfied. In fact, it seems very likely (and there is some 
evidence supporting it) that this condition is always satisfied for open billiard flows.

Given a real-valued function $q \in {\mathcal F}_{\theta}(\sAA),$ there exists a unique number $s(g) \in \R$ such that $P(-s(g) r + g) = 0.$ The following is an immediate consequence of the main result in \cite{kn:St3}.

\begin{thm}   Assume that the billiard flow $\phi_t$ over $\mt$ satisfies the conditions
{\rm (P), (NF)} and {\rm (ND)}. Let  $q \in {\mathcal F}_{\theta}(\sAA)$ be a real-valued function such that $q \circ \psi^{-1} \in \clip(U)$.
Then for any $a > 0$ there exist constants $\sigma(g) < s(q),\: C = C(a) > 0$ and 
$0 < \rho < 1$ so that for $s = \tau + \i\, t$ with $\sigma(g) \leq \tau,\: |\tau| \leq a,\: |t| \geq 1$ and 
$n = p[\log |t|] + l,\: p \in \N,\: 0 \leq l \leq [\log |t|] - 1$, for every function $v: \sAA \longrightarrow \C$ with $v \circ \psi^{-1} \in C^{\lip}(U)$ we have
\begin{equation} \label{eq:3.14}
\|(\ll_{-sr + q}^n\ v) \circ \psi^{-1}\|_{\lip, t} \leq C \rho^{p[\log|t|]} e^{l P (-\tau r + q)}\|v \circ \psi^{-1}\|_{\lip, t}.
\end{equation}
\end{thm}

As mentioned above, the conditions  ${\rm (P), (NF)}$ and ${\rm (ND)}$ are always satisfied for $N = 2$,
so  (\ref{eq:3.14}) hold for $N = 2$ without any additional assumptions. 

Now combining (\ref{eq:3.13}) and Theorem 3, we obtain estimates
for the iterations $L_{-s \tf + \tg}^n u$, provided $v = u \circ \ss$ is determined by a Lipschitz function on $R$.\\

To relate the quantities $P(-\tau r + \hat{g})$ and  $P(-\tau \tilde{f} + \tilde{g})$,  consider 
$$\underline{T}_n = \sum_{\sigma^n \iu = \iu} e^{- \tau r_n(\iu) + \hat{g}_n(\iu)} = \sum_{\sigma^n \iu = \iu}e^{-\tau d_{\gamma(\iu)} + g'_n(\iu)},$$
where $d_{\gamma(\iu)}$ is the length of the periodic trajectory $\gamma(\iu)$ determined by $\iu \in \sAA$. 
Since $\ss \circ \sigma = \sigma \circ \ss,$ we get
$$g'_n(\iu) = g_n(\ss \iu) = \tilde{g}_n(\ss \iu).$$
Setting $\ss \iu = v \in \saa$, we have $\sigma^n v = v$ and the periodic trajectory $\gamma(v)$ determined by $v$ has length 
$d_{\gamma(v)} = d_{\gamma(\iu)}.$ Thus
$$\underline{T}_n = \sum_{\sigma^n v = v} e^{-\tau d_{\gamma(v)} + \tilde{g}_n(v)} = \sum_{\sigma^n v = v} e^{-\tau \tilde{f}_n(v) + \tilde{g}_n(v)} = T_n\;.$$
On the other hand, it follows from a general property of the topological pressure that 
$$P(-\tau r + \hat{g}) = \lim_{n \to \infty} \frac{1}{n} \log \underline{T}_n$$
(see e.g. \cite{kn:R} or Theorem 20.3.7 in \cite{kn:KH} from which this property can be derived). Similarly,
$$P(-\tau \tilde{f} + \tilde{g})= \lim_{n \to \infty} \frac{1}{n} \log T_n \;,$$
so we get 
\begin{equation} \label{eq:3.15}
P(-\tau r + \hat{g}) = P(-\tau \tilde{f} + \tilde{g})\;.
\end{equation}
Introduce the number $s_0 \in \R$ such that $P(-s_0 r + \hg) = P(-s_0 \tf + \tg) = 0.$

As a consequence of (\ref{eq:3.15}), Theorem 3, Lemma 1 and Proposition 4 we obtain the following

\begin{thm}   Assume that the billiard flow $\phi_t$ over $\mt$ satisfies the conditions
{\rm (P), (NF)} and {\rm (ND)}.  Then for any $a > 0$ there exist constants $\sigma_0 < s_0$, $C' = C'(a) > 0$ and 
$0 < \rho < 1$ so that for any $s = \tau + \i\, t \in \C$ with $\tau \geq \sigma_0$, $|\tau| \leq a,\:|t| \geq 1$, any
integer $n = p[\log |t|] + l,\: p \in \N,\: 0 \leq l \leq [\log |t|] - 1$, and any function  $u: \saa \longrightarrow \R$ such that 
$u \circ \ss\circ \psi^{-1} \in \clip(U)$ we have
\begin{equation} \label{eq:3.16}
\left\|\left(L_{-s \tf + \tg }^n \, u\right)\circ \ss\circ \psi^{-1} \right\|_{\lip, t} 
\leq C' \rho^{p[\log|t|]} e^{l P (-\tau \tf + \tg)}\| u \circ \ss\circ \psi^{-1}\|_{\lip, t}.
\end{equation}
\end{thm}

\begin{proof} 
Using Theorem 3 with $q = \hg$, we find constants $\sigma_0$, $C$ and $\rho$ satisfying (3.14)
with $q = \hg$. Given $a > 0$, let $s = \tau + \i \, t \in \C$ be such that $\tau \geq \sigma_0$, $|\tau| \leq a$ and $|t| \geq 1$, and let $n = p[\log |t|] + l,\: p \in \N,\: 0 \leq l \leq [\log |t|] - 1$. 

Consider an arbitrary function $u \in C(\saa)$ such that $v = u\circ \ss\circ \psi^{-1}\in \clip(U)$.
It follows from Lemma 1 (and its proof) that there exists a constant $\Con > 0$, depending on $a$, such that
for $s\in \C$ with $|\Re(s)| \leq a$ we have $\Lip(h_s) \leq \Con\, |t|$ and 
$\Lip(d_s) \leq \Con\, |t|$. From the definitions of these functions we also have
$\|h_s\|_0 \leq \Con$, $\|d_s\|_0\leq \Con$, $\|d_s/h_s\|_0 \leq \Con$ and $\|h_s/d_s\|_0 \leq \Con$. 
This implies $\Lip(h_s/d_s) \leq \Con \, |t|$ and $\Lip(d_s/h_s) \leq \Con \, |t|$. Hence
$$\Lip( (h_s/d_s)\, v) \leq \|h_s/d_s\|_0 \, \Lip(v) + \|v\|_0\, \Lip(h_s/d_s)
\leq \Con (\Lip(v) + |t|\, \|v\|_0)\;.$$
Thus,
$$\|(h_s/d_s)\, v\|_{\lip,t} = \|(h_s/d_s)\, v\|_0 + \frac{\Lip( (h_s/d_s)\, v)}{|t|}
\leq \Con\, \left(\|v\|_0 + \frac{\Lip(v)}{|t|}\right) = \Con\, \|v\|_{\lip,t}\;.$$
Similarly, 
$$ \left\|\frac{d_s}{h_s}\; \ll_{-s \, r + \hg }^n \left( \frac{h_s}{d_s}\, v\right) \right\|_{\lip, t}
\leq \Con\, \left\|\ll_{-s \, r + \hg }^n \left( \frac{h_s}{d_s}\, v\right) \right\|_{\lip, t}\;.$$

Using the above, (3.14), (3.13) and (3.15) we get
\begin{eqnarray*}
\left\|\left(L_{-s \tf + \tg }^n \, u\right)\circ \ss\circ \psi^{-1} \right\|_{\lip, t}
& =    & \left\|\frac{d_s}{h_s}\; \ll_{-s \, r + \hg }^n \left( \frac{h_s}{d_s}\, v\right) \right\|_{\lip, t}
\leq  \Con\, \left\|\ll_{-s \, r + \hg }^n \left( \frac{h_s}{d_s}\, v\right) \right\|_{\lip, t}\\
& \leq & \Con\, \rho^{p[\log|t|]} e^{l P (-\tau r + \hg)}\, \left\| (h_s/d_s)\, v\right\|_{\lip, t}\\
& \leq & \Con\, \rho^{p[\log|t|]} e^{l P (-\tau\, \tf + \tg)}\, \| v\|_{\lip, t}\;.
\end{eqnarray*}
This proves (\ref{eq:3.16}).
\end{proof}

\section{Proofs of Theorems 1 and 2}
\renewcommand{\theequation}{\arabic{section}.\arabic{equation}}
\setcounter{equation}{0}

\def\tf{\tilde{f}}

Consider the space $\sa$ and the function $f: \Sigma_A \longrightarrow \R^{+}$ introduced in subsection 2.2. 
Let ${\mathcal M}_{\sigma}$ be the {\it space of all probability measures on $\Sigma_A$  invariant with respect to $\sigma$}. 
For a continuous function  $G: \Sigma_A \longrightarrow \R$, the {\it pressure} $P(G)$ is defined by
$$P(G) = \sup\Bigl[h(\mu) + \int G \; d\mu : \: \mu \in {\mathcal M}_{\sigma}\Bigr],$$
where $h(\mu)$ is the measure theoretic entropy of $\sigma$ with respect to $\mu$.
The measure of maximal entropy $\mu_0$ for $\sigma$ is  determined  by
$$h(\mu_0) = \sup\{h(\mu):\: \mu \in {\mathcal M}_{\sigma}\} = P(0) = h_0,$$
$h_0 > 0$ being the topological entropy of $\sigma.$ Since $\sigma$ is conjugated to the billiard ball map $B: \Lambda_{\partial K} \longrightarrow \Lambda_{\partial K}$, introduced in Sect. 2, $h_0$ coincides with the topological entropy of $B$. On the other hand, the matrix $A$ related to the symbolic codings with obstacles has an unique maximal simple eigenvalues $\lambda > 1$ and $h_0 = \log \lambda.$  Notice that if we consider, as in Section 2, the function $\tf\in {\mathcal F}_{\theta}(\sa)$ depending only on future coordinates, then $P(f) = P(\tf).$

Next, consider the space $\overline {\sa}= \sa \times \sa$ and the shift operator $\bar{\sigma}(x, y) = (\sigma x, \sigma y).$ 
Notice that $\overline{\sa}$ is a subshift of finite type with matrix
$$\bar{A} \Bigl((i, j), (i', j')\Bigr) = A(i, j) A(i', j').$$

Given a function $G: \overline{\sa} \longrightarrow \R,$ we define the {\it pressure} $P(G)$ by
$$P(G) = \sup \Bigl[ h(m) + \int G \; d m : \: m \in {\mathcal M}_{\bar{\sigma}}\Bigr],$$
${\mathcal M}_{\tilde{\sigma}}$ being the space of probability measures on $\overline {\sa}$ invariant with respect to $\tilde{\sigma}$. 
As in \cite{kn:PoS2}, the measure of maximal entropy $m_0$ for $\bar{\sigma}$ is equal to 
$\mu_0 \times \mu_0$ and  the topological entropy of $\bar{\sigma}$ is $2h_0$. 

Consider the function $F(x, y) = \tf(x) - \tf(y),\: x, y \in \saa$, and notice that $F(sR)\vert_{s=0} = 2h_0,$
$$\frac{d}{ds} F(sR) \vert_{s = 0} = \int F(x, y) \; dm_0(x,y) = \int \tf(x) d\mu_0(x) - \int \tf(y) d \mu_0(y) = 0\;.$$
Here we  used the fact that 
$$\frac{d}{ds} P(sf) \bigl\vert_{s = 0} = \int \tf(x) d\mu_0.$$

\begin{lem}There do not exist constants $a > 0$ and $c\in \R$ and continuous functions $\psi: \overline{\sa} \longrightarrow \R$
and $ M: \overline{\sa} \longrightarrow a\Z$ such that
$$F= \psi \circ \bar{\sigma} - \psi + M + c \;.$$
\end{lem}

\begin{proof} Let $d = \dist(K_1, K_2) > 0$. Assume that the above equation holds for some $a > 0$, $c\in \R$ and continuous
functions $\psi: \overline{\sa} \longrightarrow \R$ and $ M: \overline{\sa} \longrightarrow a\Z$. Fix a point
$y \in \sa$  with $\sigma^2 y = y$ corresponding  to a periodic trajectory with 2 reflection points and length $2d.$ 
Then for every $n = 2k \geq 2$ and $\sigma^n x = x,$ we get
$$F_n(x, y) = \tf_{2k}(x) - \tf_{2k}(y) = M_{2k}(x,y) + 2k c$$
so
$$\tf_{2k}(x) = M_{2k}(x, y) + 2kq,\: \forall k \in \N$$
with $q = c + d.$
Now we will exploit the construction in Lemma 5.2 in \cite{kn:St1}, where configurations 
$$\alpha_k = \{\underbrace{1,2,1,2,...,1,2}_{4k\:\:{\rm terms}},3,1\}, \: k \in \N$$
with $4k +2$ terms has been considered. Let $T_k$ be the periods of the primitive periodic rays following the configurations $\alpha_k.$ It was shown in \cite{kn:St1} that $T_k$ satisfy the estimate
$$T_{k-1} + 4d < T_k < T_{k-1} + 4d + C \delta^{2k -4}$$
with some global constants $0 < \delta < 1, \: C > 0$ independent on $k$. Then
$$T_k = a z_k + (4k + 2)q, \: T_{k-1} = a z_{k-1} + (4k -2)q, \:2d = a z_0 + 2q$$
with some integers $z_0, z_{k-1}, z_k \in \Z.$ This implies
$$z_{k-1} + 2z_0 < z_k < z_{k-1} + 2z_0  + \frac{C}{a} \delta^{2k - 4}.$$
Letting $k \to \infty$, we obtain a contradiction. 
\end{proof}

An application of Lemma 2 shows (see \cite{kn:PP}) that there exists $\beta > 0$ such that
$$\frac{d^2}{ds^2} P(sF)\vert_{s= 0} = \beta^2 > 0.$$
Repeating the proof of Lemma 1.4 and Lemma 1.5 in \cite{kn:PoS2}, we obtain the following
\begin{lem} The function $t \rightarrow e^{itF}$ has a Taylor expansion
$$e^{itF} = e^{2h_0}\Bigl(1 - \frac{\beta^2 t^2}{2} + {\mathcal O}(|t|^3)\Bigr)$$
with ${\mathcal O}(|t|^3)$ uniform on bounded intervals. Moreover, there exists a change of coordinates $v = v(t)$ such that for $|t| \leq \epsilon$ we have $e^{itR} = e^{2h_0}(1 - v^2).$
\end{lem}

We say that $x \in \saa$ is a prime point if $\sigma^n x = x$ for some $n \geq 2$ and there are no integers $m < n$ with the property $\sigma^m x = x.$ Next, as in \cite{kn:PoS2}, given a continuous non-negative function $\chi: \R \rightarrow \R$ with compact support, we introduce the function
$$\rho_N(\chi) = \sum_{\gamma, \gamma'\: \atop |\gamma|, |\gamma'| \leq N} \chi( T_{\gamma} - T_{\gamma'}).$$
Then
$$\rho_N(\chi) = \sum_{n,m = 1}^N \sum_{|\gamma| = n}\sum_{\gamma'| = m} \chi(T_{\gamma} - T_{\gamma'})$$
$$= \Xi_N(\chi) + {\mathcal O}\Bigl(\|\chi\|_{\infty} (\log N)^2 e^{3h_0 N/2}\Bigr)$$
with
$$\Xi_N(\chi) = \sum_{|\gamma| = n}\sum_{|\gamma'| = m} \frac{1}{n m} \sum_{\sigma^n x= x,\:\sigma^m y = y\atop x, y \:\:{\rm prime}\: {\rm points}} \chi(T_{\gamma} - T_{\gamma'}) $$
$$= \sum_{|\gamma| = n}\sum_{|\gamma'| = m} \frac{1}{n m} \sum_{\sigma^n x= x,\:\sigma^m y = y\atop x, y \:\:{\rm prime}\: {\rm points}} \chi(\tf_n(x) - \tf_m(y)).$$

Next the {\it proof of Theorem 1} follows without any change that of Theorem 1 in \cite{kn:PoS2} and we omit the details.
%\endofproof\\

\bs

{\bf Proof of Theorem 2}. The crucial point is an estimate for the iterations of the Ruelle operator
$$L_{it \tf -h_0} w(\xi) = \sum_{\sigma \eta = \xi} e^{it \tf(\eta) - h_0} w(\eta) = e^{-h_0} \sum_{\sigma \eta = \xi} e^{it \tf(\eta)} w(\eta) .$$
It was shown in Section 3 that $P(-\tau r-h_0) = P(-\tau \tf - h_0)$. This shows that $P(-sr - h_0)\vert_{s = 0} = h_0$,
so  for the Ruelle operator $L_{-s r - h_0}$ and $s = it,\: |t| \geq 1$ we can apply the estimates (\ref{eq:3.14}) with $v = 1.$ 
Next, the Ruelle operator $L_{it\tf - h_0}$ is conjugated to $L_{itr - h_0}$ by (\ref{eq:3.13}), hence the estimates (\ref{eq:3.14}) can be applied 
to $L_{it\tf - h_0}^n 1.$  More precisely, there exist $0 < \rho_1 < 1$ and $C > 0$ so that for $n = p[\log|t|] + l,\: 0 \leq l \leq [\log|t|] - 1$ we have
$$\|L^n_{it \tf- h_0} 1\|_{\infty} \leq C^{1/2} \rho_1^{p[\log|t|]} \: , \: n \geq 1\;.$$
This yields
$$\|L_{\pm it \tf}^n 1 \|_{\infty} \leq C^{1/2}e^{h_0 n} \rho^{p[\log|t|]/2} \: ,\: n \geq 1\;,$$
with $0 < \rho = \rho_1^{1/2} < 1.$ 

Consider the transfer operator
$$\Bigl(L_{it F} v\Bigr)(x, y) = \sum_{\bar{\sigma}(x', y') = (x, y)} e^{itF(x', y')} v(x', y').$$
Since
$$\Bigl(L^n_{it F}1\Bigr)(x_0, y_0) = \sum_{\sigma^n x = x_0, \atop \sigma^n y = y_0} e^{it F_n(x, y)} = \sum_{\sigma^n x = x_0} e^{it \tf_n(x)} \sum_{\sigma^n y = y_0} e^{-it \tf_n(y)}$$
$$= L^n_{it \tf}1(x_0) L^n_{-it \tf}1(y_0)\;,$$
it follows that
\begin{equation} \label{eq:4.1}
\|L^n_{it F} 1\|_{\infty} \leq C e^{2 h_0 n} \rho^{p[\log|t|]} \leq C e^{2 h_0 n} \rho^n \rho^{-\log|t|} \leq Ce^{2 h_0 n} \min\{\rho^n|t|^{\alpha}, 1\}
\end{equation}
with $\alpha = |\log \rho|.$ Thus we obtain the following

\begin{lem} There exists $C > 0, \: 0 < \rho <1$ and $\alpha > 0$ such that for $|t| \geq 1 > 0$ we have
$$\|L_{it F}^n 1\|_{\infty} \leq C e^{2 h_0 n} \min\{\rho^n|t|^{\alpha}, 1\}.$$
\end{lem}

\def\ss{{\mathcal S}}

Next, consider the function
$$\ss_N(t) = \sum_{n,m = 1}^N \sum_{\sigma^n x = x, \sigma^m x = x,\atop
x,\: y \:{\rm prime\:\: points}} e^{it(f^n(x) - f^m(y))}.$$
Let $\chi \in C_0^{\infty}(\R)$ be a nonnegative function and let $\hat{\chi}$ be the Fourier transform of $\chi.$
Given a sequence $\epsilon_n = {\mathcal O}(e^{-\eta n}),$ denote
%$$\hat{\chi}_N^{(z)} = e^{izu}\epsilon_N \hat{\chi}(\epsilon_N u)$$
$$A_2(N, z) = \Bigl| \int_{|t| \geq \epsilon \beta \sqrt{N}} e^{\frac{itz}{\beta \sqrt{N}}} \Bigl[e^{-2h_0 N}\Bigl(\ss_N\Bigl(\frac{t}{\beta\sqrt{N}}\Bigr) \hat{\chi}\Bigl(\frac{\epsilon_N t}{\beta \sqrt{N}}\Bigr)\Bigr)\Bigr]dt \Bigr |.$$
As in \cite{kn:PoS2}, applying Lemma 4 we obtain the following

\begin{lem} For sufficiently small $\eta > 0$ and for $N \to + \infty$ we have $\sup_{z \in \R} A_2(N, z) = 0.$
\end{lem}

The rest of the proof of Theorem 2 follows closely arguments in \cite{kn:PoS2} and we omit the details.\\

By the same arguments we get

\begin{prop} Let $\epsilon_n $ be a sequence such that $\epsilon_n = {\mathcal O}(e^{-\eta n})$, where   $\eta > 0$
is a sufficiently small constant. Then for
$$\omega(n, I_n(z)) = \#\{(\gamma, \gamma'):\: |\gamma| = |\gamma'| = n,\: z + \epsilon_n a \leq T_{\gamma} - T_{\gamma'} \leq z + \epsilon_n b\}$$
we have
$$\lim_{n \to +\infty} \sup_{z \in \R} \Bigl| \frac{n^{5/2}}{\epsilon_n e^{2 h_0 n}} \omega(n, I_n(z)) - \frac{b-a}{(2\pi)^{1/2} \beta} e^{-z^2/2\beta^2 n}\Bigr| = 0.$$
\end{prop}

\section{Separation condition for the lengths of the primitive periodic rays}
\renewcommand{\theequation}{\arabic{section}.\arabic{equation}}
\setcounter{equation}{0}
\def\tg{T_{\gamma}}
\def\etg{e^{-\delta T_{\gamma}}}
\def\pp{{\mathcal P}}
In this section we discuss some open problems related to the distribution of the lengths of primitive periodic rays. Let ${\mathcal P}$ 
be the set of all primitive periodic rays in $\Omega$. Let $\Pi$ be the set of the lengths of rays $\gamma \in \pp$  and let $\Xi$ be the 
set of periods of all periodic rays in $\Omega.$ It is known that 
\begin{equation} \label{eq:5.1}
\# \{\gamma \in \pp:\: \tg \leq x\} = \frac{e^{hx}}{hx}(1 + o(1)),\: x \to +\infty,
\end{equation}
where 
$$h = \sup_{\mu \in {\mathcal M}_{\sigma}} \frac{h(m)}{\int f d\mu} > 0$$
is the {\it topological entropy} of the open billiard flow in $\Omega$. In particular, if we have Dolgopyat type estimates for the Ruelle operator 
$\ll_{-sr}$ for $s = \tau + it,\: |t| \geq 1$ and $h - \epsilon \leq \tau \leq h$, we can obtain a sharper estimate  than (\ref{eq:5.1})(see \cite{kn:PoS1}, \cite{kn:St2}, \cite{kn:St3}) 
$$\# \{\gamma \in \pp:\: \tg \leq x\} = li(e^{hx}) + {\mathcal O}(e^{c x}),\: x \to +\infty,$$
where $li(x) = \int_2^{\infty} dt/\log t$ and $0< c < h$ . It was proved in 
\cite{kn:PS1}, Chapter 3, that for generic obstacles $K$ the periods of the primitive rays are rationally independent, that is
$$\gamma \neq \delta \Rightarrow \frac{\tg}{T_{\delta}} \notin \Q,\: \forall \gamma , \forall \delta \in \pp.$$
Thus for generic obstacles we have  $T_{\gamma} \not= T_{\delta}$ for any distinct elements $\gamma$ and $ \delta$ of $ \pp$.

In what follows we assume that the latter property is satisfied. Let $T_{\gamma}, \: \gamma \in \pp$, be ordered as a sequence
$$T_1 < T_1<....<T_n<....$$
Introduce the intervals
$$J(\gamma, \delta) = [\tg - \etg, \tg + \etg],\: \gamma \in \pp,\: \delta > 0.$$
Obviously, the number of pairs $(\tg,\: T_{\delta})$ lying in $J(\gamma, \delta)$ decreases with $\delta$.
\begin{deff}
We say that the obstacle $K$ satisfies the separation condition $(S)$ if  there exists $\delta > h > 0$ such that
\begin{equation} \label{eq:5.2}
J(\gamma, \delta) \cap \Pi = \tg,\: \forall \gamma \in \pp.
\end{equation}
We say that $K$ satisfies the condition $(S_2)$ if there exists $\delta > h > 0$ such that
\begin{equation} \label{eq:5.3}
\#\{\gamma \in \pp: \: \tg \leq x,\:  J(\gamma, \delta) \cap \Pi = \tg\} \sim e^{\frac{h}{2} x},\: x \to +\infty.
\end{equation}
\end{deff}
The condition (S) was introduced in \cite{kn:P2}, however we are not aware of any geometric conditions on $K$ that would imply (S).
The same can be said about $(S_2).$ It important to notice that Proposition 5 cannot be applied with the sequence $\epsilon_n = e^{-\delta T_n},\: \delta > h.$ 
On the other hand, it was shown in \cite{kn:P2} that under the condition (S) there exists a global constant $A_0 > 0$ such that
$$\#\{d_{\gamma} \in \Xi:\: d_{\gamma} \in J(\gamma, \delta)\} \leq A_0 \tg,\: \forall \gamma \in \pp.$$

Thus, the distribution of all periods in $\Xi$ is such that we have no {\em clustering} of periods with big density  in $J(\gamma, \delta).$ 
It is clear that the condition $(S_2)$ is much weaker than $(S)$. \\

Let $\chi(t) \in C_0^{\infty}(-1, 1)$ be a positive function such that $\chi(t) = 1$ for $|t| \leq \epsilon_0 < 1/2$ with 
Fourier transform $\hat{\chi}(\xi) \geq 0,\: \forall \xi \in \R.$ Consider the function
$\varphi_j(t) = \chi(e^{\delta T_j} (t - T_j)).$
For the proof of the so called {\it Modified Lax-Phillips Conjecture (MLPC)} for the Dirichlet problem in $\R \times \Omega$ 
(see \cite{kn:I2}) it is necessary to find $\delta > h$ and a sequence $j \to + \infty$ so that
\begin{equation} \label{eq:5.4}
\Bigl|\sum_{\gamma} (-1)^{|\gamma|} \tg |\det(I - P_{\gamma})|^{-1/2} \varphi_j(d_{\gamma})\Bigr| \geq \eta_0 e^{-\eta T_j}
\end{equation}
with $\eta_0 > 0,\:\eta > 0$ independent on $j$. Here $P_{\gamma}$ is the {\it linear Poincar\'e map} related to $\gamma$ introduced in Sect. 1. 
It is easy to see that under the condition $(S_2)$ we can arrange (\ref{eq:5.4})  for a suitable sequence $j \to + \infty$, so (MLPC) follows from $(S_2)$. 

Indeed, assuming $(S_2)$, consider the number of iterated periodic rays $\mu \notin \pp$ with lengths
$d_{\mu} \in J(\gamma, \delta)$, provided $\gamma \in \pp,\:\tg \leq x.$ If $d_{\mu} = k T_{\nu},\: k \geq 2,$ then 
$2 \leq k \leq \frac{x + 1}{d_0},$ where $d_0 = 2\min_{i \not= j} {\rm dist} \: (K_i, K_j).$ Thus, using (\ref{eq:5.1}), 
the number of such non primitive periodic rays $\mu$ for $x \geq M_0$ large enough is not greater than
$$\frac{C}{h(x+1)} \sum_{k = 2}^{(x+1)/d_0} ke^{\frac{hx}{k}} \leq \frac{C_1 e^{\frac{h}{2}x}}{x+1}\Bigl( 1 + e^{-\frac{h}{6}x}\sum_{k=3}^{(x+1)/d_0}\frac{ k}{2}\Bigr) \leq \frac{1}{2}e^{\frac{h}{2}x}.$$
Consequently, taking into account (\ref{eq:5.3}), it is possible to find an infinite number of intervals
$J(\gamma_j, \delta)$ with $\gamma_j \in \pp$ such that $T_{\gamma_j} \to +\infty$ and
$$J(\gamma_j, \delta) \cap \Xi = T_{\gamma_j}\;.$$
For such rays the sum on the left hand side of (\ref{eq:5.4}) is reduced to the term
$T_{\gamma_j}|\det(I - P_{\gamma_j})|^{-1/2}$ and the inequality (\ref{eq:5.4}) follows from the
estimate
$$|\det(I - P_{\gamma_j})| \leq e^{c T_{\gamma_j}}$$
with some global constant $c > 0$ (see for instance, Appendix A.1 in \cite{kn:P1}). 
By the argument of Ikawa \cite{kn:I2} we conclude that (\ref{eq:5.4}) implies (MLPC).\\

The above analysis shows that $(S_2)$ can be replaced by a  weaker condition in order to satisfy (\ref{eq:5.4}).
Let $\pp_e$ (resp. $\pp_o$) be the set of primitive periodic rays with even (resp. odd) number of reflections and let
$\Pi_e$ (resp. $\Pi_o$) be the set of periods of $\gamma \in \pp_e$ (resp. $\gamma \in \pp_o$). We have (see \cite{kn:G}, \cite{kn:X1}, \cite{kn:X2}) the following analog of (\ref{eq:5.1})
$$\# \{ \gamma \in \pp_e: \: \tg \leq x\} \sim \frac{e^{hx}}{2hx},\: x \to +\infty,$$
$$\# \{ \gamma \in \pp_o: \: \tg \leq x\} \sim \frac{e^{hx}}{2hx},\: x \to +\infty.$$
\begin{deff} We say the obstacle $K$ satisfies the condition $(S_3)$ if there exists $\delta > h$ such that
\begin{equation}
\# \{\gamma \in \pp_e: \: \tg \leq x,\: J(\gamma, \delta) \cap \Pi_o = \emptyset\} \sim e^{\frac{h}{3}x},\: x \to + \infty.
\end{equation}
\end{deff}
Under the condition $(S_3)$ in every interval $J(\gamma, \delta),\: \gamma \in \pp_e,$ we may have an arbitrary {\it clustering} of  periods in $\Pi_e$ and this leads to a sum of terms with positive signs in (\ref{eq:5.4}). Assuming $(S_3)$ fulfilled, it is not hard to 
deduce  that there exists 
a sequence of $\gamma_j \in \pp_e,\: T_{\gamma_j} \to +\infty,$ such that
$$J(\gamma_j, \delta) \cap \Xi \subset (\Pi_e \cup 2\Pi),\: \forall j \in \N\;,$$
and the latter implies (MLPC). We leave the details to the reader.\\

It is an interesting open problem to investigate if the conditions $(S_2),\: (S_3)$ or some other similar condition are fulfilled.

\end{document}